\documentclass{article}
\usepackage[margin=2.5cm]{geometry}

\RequirePackage{amsthm,amsmath,amsfonts,amssymb}
\RequirePackage[numbers]{natbib}

\usepackage{dirtytalk}
\usepackage{amsmath}
\usepackage{amssymb}
\usepackage{amsfonts}
\usepackage{amsthm}
\usepackage{bm}
\usepackage{cleveref}
\usepackage{dsfont}
\usepackage{ytableau}
\usepackage{dsfont}
\usepackage{enumitem}
\usepackage{booktabs}

\usepackage{titlesec}

\titleformat*{\paragraph}{\small\bfseries}

\newcommand{\vocab}{\textit}

\newcommand{\PP}{\mathrm{Pr}}

\newcommand{\CC}{\mathbb{C}}
\newcommand{\TT}{\mathcal{T}}

\newcommand{\norm}[1]{\left\lVert#1\right\rVert}
\newcommand{\bigO}[1]{O\left(\frac{1}{#1}\right)}
\DeclareMathOperator{\Res}{Res}

\newcommand{\OSC}{P}
\DeclareMathOperator{\eval}{eval}

\DeclareMathOperator{\eig}{eig}

\DeclareMathOperator{\Var}{Var}
\DeclareMathOperator{\EE}{E}
\DeclareMathOperator{\SYT}{SYT}
\DeclareMathOperator{\TV}{T.V.}

\renewcommand{\SS}{\mathfrak{S}}

\usepackage{xcolor}
\newcommand{\new}[1]{#1}

\usepackage{booktabs}

\newtheorem{theorem}{Theorem}
\newtheorem{corollary}[theorem]{Corollary}
\newtheorem{proposition}[theorem]{Proposition}
\newtheorem{lemma}[theorem]{Lemma}

\newtheorem{definition}[theorem]{Definition}

\title{Mixing times of one-sided $k$-transposition shuffles}
\author{Evita Nestoridi\thanks{Department of Mathematics, Stony Brook University (evrydiki.nestoridi@stonybrook.edu). Funded by DMS-2346986 and MPS-TSM-00007955.}\qquad\quad Kenny Peng\thanks{Department of Computer Science, Cornell University (kennypeng@cs.cornell.edu)} \qquad\quad Bryan Wong\thanks{Department of Mathematics, Stony Brook University (bryan.wong@stonybrook.edu)}}
\date{}

\begin{document}

\maketitle

\begin{abstract}
We study mixing times of the one-sided $k$-transposition shuffle. We prove that this shuffle mixes relatively slowly, even for $k$ big. Using the recent ``lifting eigenvectors'' technique of Dieker and Saliola \cite{dieker} and applying the $\ell^2$ bound, we prove different mixing behaviors and explore the occurrence of cutoff depending on $k$.
\end{abstract}

\section{Introduction}

Diagonalizing the transition matrix of a reversible Markov chain is extremely powerful when wanting to prove that the Markov chain exhibits the cutoff phenomenon. The first technique for diagonalizing the transition matrix of a random walk on the Cayley graph of a finite group $G$ was introduced by Diaconis and Shahshahani \cite{Diaconis1981GeneratingAR}. The technique, which relies on Schur's lemma, requires understanding of the representation and character theory of $G$, and has been applied for many random walks on groups \cite{BH, BJ, Hild, NN, Ros}. 

Cases where the generating set is not a conjugacy class are much more challenging. An early example is the case of star transpositions, which was diagonalized by Flatto, Odlyzko and Wales \cite{FOW}. Diaconis \cite{PD} analyzed this diagonalization to show cutoff at $n \log n$. In a recent breakthrough, Dieker and Saliola \cite{dieker} introduced a new technique to diagonalize the random-to-random shuffle.  The proof of cutoff for random-to-random was completed by Bernstein and the first author's eigenvalue analysis \cite{BN} and Subag's lower bound analysis \cite{Subag}. 

Another development was studying the \textit{one-sided transposition shuffle} on $n$ cards, during which different transpositions are assigned different weights. Bate, Connor and Matheau--Raven \cite{raven} diagonalized this shuffle and proved that it exhibits cutoff at $n \log n$. 
One step of this shuffle consists of choosing a position $R$ uniformly at random, choosing a position $L$ from $\{1,2,\cdots, R\}$ uniformly at random, and then performing the transposition $(RL)$. Here, we introduce a generalization called the \textit{one-sided $k-$transposition shuffle}. As before, we choose a position $R$ uniformly at random, except now, we pick $k$ positions $L_1,\cdots, L_{k}$ (not necessarily distinct) uniformly at random from $\{1,2,\cdots,R\},$ and perform the permutation $(RL_{k})(RL_{k-1})\cdots (RL_1).$ These products can give rise to many types of permutations of varying weights.

Let $\OSC_{n,k}$ denote the transition matrix of the \textit{one-sided $k-$transposition shuffle} on $n$ cards and let $U$ denote the uniform measure on $S_n$. We define the total variation- and $\ell^2$- distance between $P_{n,k}$ and $U$ as follows:
\begin{align*}
d^{(n,k)}(t)&:= \norm{\OSC_{n,k}^t - U}_{\TV}=   \frac{1}{2} \sum_{y \in S_n} \vert \OSC_{n,k}^t(id,y)- U(y) \vert 
\\
\norm{\frac{\OSC_{n,k}^t}{U}- 1}_2^2 &:= \left(\sum_{y\in S_n} \left|\frac{P_{n,k}^t(id,y)}{U(y)}\right| - 1\right),
\end{align*}
where $\OSC_{n,k}^t(x,y)$ is the probability of moving from $x$ to $y$ after $t$ steps of the shuffle. We note that since $P_{n,k}$ is transitive, we can without loss of generality start the card shuffle at the identity element $id$ of $S_n$. 
The mixing time of $\OSC_{n,k}$ is defined as
\[t_{\textup{mix}}(\varepsilon)= \min \lbrace t: d^{(n,k)}(t)\leq \varepsilon \rbrace.\]
A shuffle exhibits cutoff if as $n$ grows, the total variation distance is almost equal to one and then suddenly drops and approaches zero. More formally, $P_{n,k}$ is said to exhibit cutoff at time $t_{n,k}$ with window $w_{n,k}=o(t_{n,k})$ if and only if
\[\lim_{c \rightarrow \infty} \lim_{n \rightarrow \infty} d^{(n,k)}(t_{n,k}-cw_{n,k})= 1 \quad \mbox{and} \quad \lim_{c \rightarrow \infty} \lim_{n \rightarrow \infty} d^{(n,k)}(t_{n,k}+cw_{n,k})= 0.\]
We may analogously define the $\ell^2$-cutoff. Salez \cite{JS} gives breakthrough developments and a nice exposition on the history of cutoff. In this paper, we present a series of results involving the mixing time of $\OSC_{n,k}$ for different regimes of $k$:

\begin{itemize}
\item[R1.] When $k = n^{o(1)}$, $\OSC_{n,k}$ exhibits total-variation cutoff at $t = n\log n.$

\item[R2.] When $k = n^{\gamma}$ for $\gamma\in(0,1]$ and odd, we have \mbox{$(1-\gamma)n\log n\le t_{\textup{mix}}(1/2)\le (1 - \frac{\gamma}{2})n\log n$} with $\ell^2$ cutoff at $t = (1 - \frac{\gamma}{2})n\log n.$

\item[R3.] When $k = \Omega(n\log n)$, $\OSC_{n,k}$ mixes in order $n$ steps without cutoff.
\end{itemize}

In particular, we observe the surprising fact that even as $k$ increases initially, the mixing time does not change. This stands in sharp contrast to other classes of shuffles, such as the $k-$cycle shuffle \cite{BH, BSZ}, and the conjugacy class random walks \cite{BJ}. In fact, this turns out to be a very slow shuffle even when $k$ is very big (e.g., when $k= n \log n$). This is unlike other non-local shuffles such as the riffle shuffles, which Bayer and Diaconis proved mixes in $\frac{3}{2}\log_2 n$ steps \cite{BD}. 

We now define the shuffle more carefully.
\new{\begin{definition}[One-sided $k-$transposition shuffle]
The one-sided $k-$transposition shuffle $\OSC_{n,k}$ is the ergodic random walk on $S_n$ generated by the following probability distribution:
$$\OSC_{n,k}(\tau) = \sum_{\substack{1\le i_1,\cdots,i_k\le j\le n:\\ \tau=(j;i_1,\cdots,i_k)}}
\frac{1}{nj^k}.$$
where we set the notation
$$(j;i_1,\cdots,i_k) := (ji_k)\cdots (ji_1),$$
the composition of $k$ transpositions with a common element. Note that there are some permutations $\tau$ that cannot be expressed in this form (and thus occur with probability $0$ in the shuffle), and other permutations that can be described by multiple of these $(j;i_1,\cdots,i_k).$
\end{definition}}
Our strategy is to calculate the eigenvalues of $P_{n,k}$ using the lifting eigenvectors method. This technique, pioneered by Dieker and Saliola \cite{dieker}, allows us to compute the eigenvalues of the $P_{n+1,k}$ from the $P_{n,k}$.

Once we diagonalize $P_{n,k}$, we will leverage the following classical bound, which connects the eigenvalues 
$1=\beta_1>\beta_2\ge \cdots \ge \beta_{n!}>-1$
of $P_{n,k}$ to its total variation distance from the stationary distribution:
\begin{equation}\label{lbound}
    4\norm{\OSC_{n,k}^t - U}_{\TV}^2 \le \sum_{i\neq 1}\beta_i^{2t} = \norm{\frac{\OSC_{n,k}^t}{U} - 1}_2^2.
\end{equation}
We now state our main results. The first result discusses a general upper bound for the mixing time, which turns out to be sharp for $k=n^{o(1)}$. We also provide a better bound for the case $k=n^\gamma$ with $\gamma\in (0,1]$, which turns out to be sharp for the $\ell^2$ norm.
\begin{theorem}[Upper bounds on total variation and $\ell^2$ distance]\label{thm:general-upper}
\mbox{}
\begin{enumerate}
\item[(i)] For odd $k\ge 1$, when $t = n\log n + cn$, $c > 0$, for $n$ sufficiently large, there exists a universal constant $A$ such that
$$4\norm{\OSC_{n,k}^t - U}_{\TV}^2 \le \norm{\frac{\OSC_{n,k}^t}{U} - 1}_2^2 < Ae^{-c}.$$

\item[(ii)] For even $k \ge 4$, when $t = n\log(n) + cn$, $c > 1$, for $n$ sufficiently large, there exists a universal constant $B$ such that

$$4\norm{\OSC_{n,k}^t - U}_{\TV}^2 \le \norm{\frac{\OSC_{n,k}^t}{U} - 1}_2^2 < Be^{-c}$$

\item[(iii)] For $k=2$, when $t = \frac{3}{2}n\log(n) + cn$, $c > 1$, for $n$ sufficiently large, there exists a universal constant $C$ such that 

$$4\norm{\OSC_{n,k}^t - U}_{\TV}^2 \le \norm{\frac{\OSC_{n,k}^t}{U} - 1}_2^2 < Ce^{-c}$$

\item[(iv)] For odd $k=n^\gamma$ with $\gamma\in (0,1)$, when $t = (1 - \frac{\gamma}{2})n\log n + cn$ and $c>3$ for $n$ sufficiently large,
$$4\norm{\OSC_{n,k}^t - U}_{\TV}^2\le \norm{\frac{\OSC_{n,k}^t}{U} - 1}_2^2 < 10e^{-c}.$$
\end{enumerate}
\end{theorem}
The following theorem discusses the $\ell^2 $ mixing time, which in combination with the previous theorem shows (R2).
\begin{theorem}[Lower bounds on $\ell^2$ distance]\label{thm:l2-lower-bounds}
\mbox{}
\begin{enumerate}
\item[(i)] For $k\ge 1$, when $t = \frac{1}{2}n\log n - cn$ for $n$ sufficiently large,
$$\norm{\frac{\OSC_{n,k}^t}{U} - 1}_2 > \frac{1}{2}e^c.$$
\item[(ii)] For $k = n^\gamma$ with $\gamma \in (0,1],$ when $t = (1 - \frac{\gamma}{2})n\log n - \frac{1}{2}n\log \log n - cn$ for $n$ sufficiently large,
$$\norm{\frac{\OSC_{n,k}^t}{U} - 1}_2 > e^c.$$
\end{enumerate}
\end{theorem}
The next theorem discusses a lower bound on the total variation distance. In combination with Theorem \ref{thm:general-upper}, it concludes cutoff for the the case $k=n^{o(1)}$ as described in (R1).
\begin{theorem}[Lower bound on total variation distance]\label{thm:lower}
For $k = o\left(\frac{n}{\log n}\right)$, when $t=n\log(n/k) - n\log \log n - cn,$
$$\liminf_{n\rightarrow\infty} \norm{\OSC_{n,k}^t - U}_{\TV} \ge 1 - \frac{\pi^2}{6(c-4)^2}.$$
\end{theorem}
The following theorem discusses mixing times for the case where $k$ is big, giving the results in (R3).
\begin{theorem}[Mixing time for especially big $k$]\label{thm:large-k}
For $k = \Omega (n \log n) $, we have
$t_{\text{mix}}(\varepsilon) =\Theta(n)$. For $k \in [n, n \log n]$, we have that $t_{\text{mix}}(\varepsilon) =O \left( \frac{n^2\log n}{k}\right)$. We also have that $t_{\text{mix}}(\varepsilon) = \Omega(n)$ for every $k \geq 1$.
\end{theorem}

We now outline the remainder of the paper. In Section \ref{sec:prelims}, we give the definitions needed to describe the spectrum of $P_{n,k}$. Section \ref{sec:eig-general} contains the proof of Theorem \ref{thm:general-upper}(i). Theorem \ref{thm:general-upper}(ii) is proven is Section \ref{sec:eig-main}. In Section \ref{sec:l2}, we present the $\ell^2$ lower bounds summarized in Theorem \ref{thm:l2-lower-bounds}. The total variation bound of Theorem \ref{thm:lower} can be found in Section \ref{sec:lb}. Theorem \ref{thm:large-k} is proved in Section \ref{sec:lk}.

We conclude our introduction by suggesting a few interesting open questions. We first suggest the question of if there is total-variation cutoff in (R2) and whether it coincides with the $\ell^2$-cutoff. 

Another natural question to ask concerns the limit profile of the shuffle. The limit profile (if it exists) is defined as the function
\[\Phi(c)= \lim_{n \rightarrow \infty} d^{(n,k)}(t_{n,k}+cw_{n,k}),\]
where $c \in \mathbb{R}$ and $t_{n,k}, w_{n,k}$ are the cutoff time of the Markov chain and the corresponding window. There are a few examples of famous Markov chains whose limit profile has been determined \cite{BD,BuN,LP}. Recently, there has been exciting progress on developing techniques to determine limit profiles \cite{EN, NS, Teyssier}, which work well for conjugacy class invariant random walks or random walks where we have knowledge of the eigenvalues and the eigenvectors of the transition matrix. It would be very interesting to determine the limit profile of $\OSC_{n,k}$ for $k=n^{o(1)}$, or simply $k=1$, since it cannot be studied by the already existing techniques and could lead to developing new ones. 

\section{Preliminaries: Partitions and Standard Young Tableaux}\label{sec:prelims}

In this section, we introduce several standard definitions involving partitions and standard Young tableaux. A \vocab{partition} $\lambda$ of an integer $n$ is a tuple $(\lambda_1,\cdots,\lambda_r)$ of positive integers summing to $n$ such that $\lambda_1\ge \cdots \ge \lambda_{r}.$ We will write $\lambda\vdash n$ to indicate that $\lambda$ is a partition of $n$, and let $l(\lambda):=r$ denote the \vocab{length} of $\lambda$, i.e. the number of parts of $\lambda$.

We may associate a partition $\lambda$ to its \vocab{Young diagram}, which has $l(\lambda)$ rows of left-aligned boxes, such that from top to bottom the rows have $\lambda_1, \lambda_2, \cdots, \lambda_{r}$ boxes. For example, the partition $(6,4,2)\vdash 12$ corresponds to the following Young diagram:
$$\ydiagram{6,4,2}$$
We will often refer to a partition $\lambda$ and its diagram interchangeably. For example, for partitions $\lambda,\mu\vdash n,$ we write that $\lambda\trianglerighteq \mu$ (``$\lambda$ \vocab{dominates} $\mu$'') if $\mu$ can be obtained by moving boxes in $\lambda$ down and to the left.

Given $\lambda\vdash n,$ we can create a \vocab{standard Young tableau} of \vocab{shape} $\lambda$ by placing each of the numbers $1,2,\cdots, n$ in the diagram of $\lambda$ such that the numbers are strictly increasing across each row and down each column. For example, the following is a standard Young tableau of shape $(6,4,2)\vdash 12:$
$$\ytableaushort{123567,489{10},{11}{12}}$$
For a standard Young tableau $T$, we let $T(i,j)$ denote the number in the $i-$th row and $j-$th column. For $T$ given above, $T(2,1)=4.$

For $\lambda\vdash n,$ we denote $\SYT(\lambda)$ as the set of all standard Young tableaux of shape $\lambda.$ We let $d_\lambda := |\SYT(\lambda)|$ indicate the \vocab{dimension} of $\lambda$. Calculating $d_\lambda$ is challenging in general, and is given by the famous hook-length formula. For our purposes, the following bound---which we will recall later on---suffices.
\begin{proposition}[Corollary 2 in \cite{Diaconis1981GeneratingAR}]\label{pre:dlambda-bound}
Let $\lambda_1$ denote the first part of a partition $\lambda\vdash n$. Then
$$\sum_{\substack{\lambda\vdash n\\ \lambda_1=n-m}}d_\lambda^2 < \frac{n^{2m}}{m!}.$$
\end{proposition}

As we will see in the next section, the standard Young tableaux index the eigenvalues of the one-sided $k-$transposition shuffle, thus playing an essential role in our analysis.

In this section, we analyze the eigenvalues of the one-sided $k-$transposition shuffle, giving bounds that will help obtain results about the shuffle's mixing time in several regimes. The following result shows that these eigenvalues are indexed by standard Young tableaux. We defer the proof, which uses tools from representation theory, to the appendix.

\begin{theorem}\label{thm:eig-syt}
The eigenvalues of $P_{n,k}$ are labeled by standard Young tableaux of size $n$, where $T(i,j)$ is the entry in the $(i,j)$ coordinate of the standard tableau $T$, and
\begin{equation}\label{eq:eig-syt} \eig(T) = \frac{1}{n}\sum_{(i,j) \in T} \left(\frac{1 + j - i}{T(i,j)}\right)^k. \end{equation}
Furthermore, the eigenvalue $\eig(T)$ corresponding to a standard Young tableau $T$ of shape $\lambda$ appears $d_\lambda$ times.
\end{theorem}
Applying \eqref{lbound}, this reveals the following bound, which is central to our analysis.
\begin{equation}\label{prop:tv-eig-bound}
4\norm{\OSC_{n,k}^t - U}_{\TV}^2 \le \sum_{\substack{\lambda\vdash n\\ \lambda \neq (n)}}d_\lambda \sum_{T\in SYT(\lambda)}\eig_k(T)^{2t}.
\end{equation}

In \Cref{tab:eigs}, we provide eigenvalues of $P_{4,k}$ in the cases $k=1, 2, 3,$ and $4,$ which are labeled by the standard Young tableaux of size $4$. Note that some eigenvalues are negative. Furthermore, \Cref{thm:eig-syt} implies that the eigenvalues of $P_{n,k}$ are rational.

\begin{table}[h]
  \centering
  \ytableausetup{smalltableaux}
  \caption{$\eig_k(T)$ for $T\in \SYT(\lambda)$ and $|\lambda|=4$, rounded to three decimal places}
  \label{tab:eigs}
  \begin{tabular}{ccccccccccc}
    \toprule
    &
    \ytableaushort{1234} &
    \ytableaushort{123,4} &
    \ytableaushort{124,3} &
    \ytableaushort{134,2} &
    \ytableaushort{12,34} &
    \ytableaushort{13,24} &
    \ytableaushort{12,3,4} &
    \ytableaushort{13,2,4} &
    \ytableaushort{14,2,3} & 
    \ytableaushort{1,2,3,4} \\
    \midrule
    $k=1$ &
    1.0 &
    0.75 &
    0.688 &
    0.604 &
    0.563 &
    0.479 &
    0.438 &
    0.354 &
    0.292 &
    0.042 \\
    $k=2$ &
    1.0 &
    0.75 &
    0.641 &
    0.502 &
    0.516 &
    0.377 &
    0.516 &
    0.377 &
    0.340 &
    0.340 \\
    $k=3$ &
    1.0 &
    0.75 &
    0.605 &
    0.430 &
    0.504 &
    0.328 &
    0.496 &
    0.320 &
    0.272 &
    0.201 \\
    $k=4$ &
    1.0 &
    0.75 &
    0.579 &
    0.379 &
    0.501 &
    0.300 &
    0.501 &
    0.300 &
    0.269 &
    0.269 \\
    \bottomrule
  \end{tabular}
\end{table}

\section{Upper bound in the general case}\label{sec:eig-general}

In this section, we show \Cref{thm:general-upper}(i), which states that $\OSC_{n,k}$ mixes in at most $n\log n + cn$ time for all positive integers $k \neq 2.$ This makes sense intuitively, as the case $k=1$ was shown in \cite{raven}, and we would expect more transpositions at each step to only speed up the shuffle. For odd $k$, it is easy to prove the result from the case $k=1$, where as the result for $k=2$ needs slightly more work.\par
We refer to the analysis of section 2.2 in \cite{raven} to show the following lemma is sufficient to prove Theorem \ref{thm:general-upper}.

\begin{lemma}\label{lem:raven}
Assume $k\ge 1$ and $n$ is sufficiently large. Then for $T\in \SYT(\lambda)$ where $\lambda \vdash n$ with $\lambda_1 = n-r,$
$$\left|\eig(T)\right|\le 
\begin{cases} 
      \frac{n-r}{n} + \frac{1}{n}\sum_{j=1}^{r} \frac{j-1}{n-r+j}, & \textup{if }  m\le \frac{n}{4} \\
      \frac{n-r}{n} + \frac{1}{n}\sum_{j=2}^{n-r} \frac{j-1}{n-r+j} + \frac{n - 2(n-r)}{3n}, & \textup{otherwise.}
   \end{cases}$$
\end{lemma}

The following lemma lets us characterize how $\eig(T)$ changes.

\begin{lemma}\label{lem:difference} Suppose $T \in \text{SYT}(\lambda), \lambda \vdash n$. Say that $(i_1,j_1),(i_2,j_2) \in T$ are such that $T(i_1,j_1) < T(i_2,j_2)$. Let $S$ be the tableau (not necessarily standard) obtained by switching the numbers in these coordinates, then for odd $k$,

\begin{equation}\eig(S) - \eig(T) \begin{cases} \geq 0\,\,\,\,\text{ if } (i_1 - i_2) + (j_2 - j_1) \geq 0 \\ < 0\,\,\,\, \text{ if } (i_1 - i_2) + (j_2 - j_1) < 0 \end{cases}, \end{equation}

\end{lemma}
\begin{proof}
Since the only difference between $S,T$ is at the coordinates $(i_1,j_1),(i_2,j_2)$, we have that
\begin{equation*}\begin{split} \eig(S) - \eig(T) &= \left(\frac{j_2 - i_2 + 1}{T(i_1,j_1)}\right)^k + \left(\frac{j_1 - i_1 + 1}{T(i_2,j_2)}\right)^k - \left(\frac{j_1 - i_1 + 1}{T(i_1,j_1)}\right)^k - \left(\frac{j_2 - i_2 + 1}{T(i_2,j_2)}\right)^k\\
&= \left(\frac{1}{T(i_1,j_1)^k} - \frac{1}{T(i_2,j_2)^k}\right)\left((j_2 - i_2 + 1)^k - (j_1-i_1+1)^k\right).\end{split}\end{equation*}
Then the sign of $\eig(S)- \eig(T)$ depends only on the sign of $(j_2-i_2+1)^k - (j_1-i_1+1)^k$, which gives the result after breaking into the cases when $k$ is even and odd.
\end{proof}
Let $ \lambda \vdash n$ and let $T^\rightarrow_\lambda$ denote the SYT with numbers filled into $\lambda$ left to right. Similarly, let $T^\downarrow_\lambda$ denote the SYT with numbers filled from top to bottom. The following lemmas are generalizations of Lemmas 9 and 12 of \cite{raven}.
\begin{lemma}\label{lem:oddfill}
Suppose that $k$ is odd, and $\lambda \vdash n$. Then for any $T \in \text{SYT}(\lambda)$
$$\eig(T^\downarrow_\lambda) \le \eig(T) \le \eig(T^\rightarrow_\lambda).$$
\end{lemma}
\begin{proof}
Reading across the rows of $T$, starting from the top, find the first box where $T$ and $T_\lambda^\rightarrow$ are different, denote its coordinates as $(i,j)$. The number $T(i,j) - 1$ must occur in a box below and strictly left of $(i,j)$. By switching the entries $T(i,j)$ and $T(i,j) - 1$, we obtain another standard tableau such that its associated eigenvalue is larger due by Lemma \ref{lem:difference}. The statement of the lemma follows by induction on the number of boxes which $T$ and $T_\lambda^\rightarrow$ agree at.
\end{proof}

There is an analogous monotonicity statement in terms of comparing eigenvalues corresponding to different partitions.
\begin{lemma}\label{lem:oddshape}
Suppose that $k$ is odd. If $\lambda,\mu \vdash n$, $\lambda \trianglerighteq \mu$, then 
$$\eig(T^\rightarrow_\lambda) \ge \eig(T^\rightarrow_\mu)\,\,\,\,\, \textup{ and }
\eig(T^\downarrow_\lambda) \ge \eig(T^\downarrow_\mu).$$
\end{lemma}
\begin{proof}
The strategy of this proof follows very closely to the proof of lemma 12 of \cite{raven}. We prove that this $\eig(T^\rightarrow_\lambda) \le \eig(T^\rightarrow_\mu)$, where $\mu$ is obtained from $\lambda$ by moving a single box from the end of a row at coordinates $(a,\lambda_a)$ to the end of a lower row with coordinates $(b,\lambda_b+1)$. We allow for the case that $b = l(\lambda) + 1$ and $\lambda_b = 0$. The change in the eigenvalue is the following:

\begin{equation*} \begin{split} 
n(\eig(\lambda) - \eig(\mu)) &= \left(\frac{\lambda_a - a + 1}{T_\lambda^\rightarrow(a,\lambda_a)}\right)^k - \left(\frac{\lambda_b - b + 2}{T_\mu^\rightarrow(b,\lambda_b+1)}\right)^k \\ 
&+ \sum_{\substack{(i,j) \in \lambda \cap \mu \\ a < i \le b}} \left(\frac{1}{T^\rightarrow_\lambda(i,j)^k} + \frac{1}{T^\rightarrow_\mu(i,j)^k}\right)(j-i+1)^k \\
&\ge \left(\frac{\lambda_a - a + 1}{T_\lambda^\rightarrow(a,\lambda_a)}\right)^k - \left(\frac{\lambda_b - b + 1}{T_\mu^\rightarrow(b,\lambda_b+1)}\right)^k \\
&= (\lambda_a - a + 1)^k\left(\frac{1}{T^\rightarrow_\lambda(a,\lambda_a)^k} + \frac{1}{T^\rightarrow_\mu(b,\lambda_b+1)^k}\right)\\
&= \frac{(\lambda_a - a +1)^k - (\lambda_b - b + 2)^k}{T(b,\lambda_b+1)^k}.
\end{split}
\end{equation*}

This final expression is positive, since $b > a$ and $\lambda_a > \lambda_b$.

\end{proof}

Using the bound of Lemma \ref{lem:oddfill} we get the following.
\begin{lemma}\label{lem:biglam} 
For odd $k$, the $\ell_2$ bound becomes
$$4\norm{\OSC_{n,k}^t - U}_{\TV}^2 \le \eig(T^\rightarrow_{(1^n)})^{2t} + 2\sum_{\substack{\lambda: \eig(T^\rightarrow_\lambda) \ge 0 \\ \lambda \neq (n)}} d_\lambda^2 \eig(T^\rightarrow_\lambda)^{2t}.$$
\end{lemma}
\begin{proof}
Notice that for any tableau $T \in \text{SYT}(\lambda)$,
$$\eig(T) + \eig(T') = \sum_{(i,j) \in T} \frac{(j-i+1)^k - (j-i-1)^k}{T(i,j)} > 0,$$
where $T' \in \text{SYT}(\lambda')$ and $T'(i,j) = T(j,i)$. This implies that if $\eig(T) < 0$, then $\eig(T') > |\eig(T)|$. By construction we have that $(T_\lambda^\downarrow)' = T_{\lambda'}^\rightarrow$. It follows that if $\eig(T_\lambda^\rightarrow) < 0$, then 
$$|\eig(T_\lambda^\downarrow)| < \eig(T_{\lambda'}^\rightarrow).$$ 
The rest of this proof follows exactly the same as in section 2.2 of \cite{raven}. 
\end{proof}

For odd $k$, our analysis is almost the same as in \cite{raven}, we only need justify that each part of the analysis is not changed by raising to the power of $k$. Lemmas 9, 10 and 11 show that for $\lambda$ with $\lambda_1 = n - r \geq 3n/4$, the maximal eigenvalue appearing among all such $T$ with shape $\lambda$ is $T^\rightarrow_{(n-r,r)}$, the tableau which fills row one with $1,2,\ldots,n-r$, and row two with $n-r+1, n-r+2, \ldots, n$. Then from the formula for $\eig(T)$,
$$\eig(T) \leq \eig(T^\rightarrow_{(n-r,r)}) = \frac{n-r}{n} + \frac{1}{n}\sum_{j=1}^{r} \left(\frac{j-1}{n-r+j}\right)^k \leq \frac{n-r}{n} + \frac{1}{n}\sum_{j=1}^{r} \frac{j-1}{n-r+j}.$$
If $n-r < \frac{n}{2}$, then for any $T$ with shape $\lambda$, $\eig(T) \leq \eig(T^\rightarrow_{(n-r,\star)})$, where $(n-r,\star)$ is the partition with as many parts equal to $n-r$ as possible, and $T^\rightarrow_{(n-r,\star)}$ fills the rows of this diagram from left to right, top to bottom. It is apparent that $\frac{j-i+1}{T^\rightarrow_{(n-r,\star)}(i,j)} \leq \frac{1}{3}$. With this

\begin{equation}\label{eq:smallpart}\begin{split} \eig(T) \leq \eig(T^\rightarrow_{(n-r,\star)}) &= \frac{n-r}{n} + \frac{1}{n}\sum_{j=1}^{n-r} \left(\frac{j-1}{n-r+j}\right)^k + \frac{1}{n}\sum_{\substack{(i,j) \\ i > 2}} \left(\frac{j-i+1}{T^\rightarrow_{(n-r,\star)}(i,j)}\right)^k \\ &\leq \frac{n-r}{n} + \frac{1}{n}\sum_{j=1}^{n-r} \frac{j-1}{n-r+j} + \frac{n - 2(n-r)}{3n}.\end{split}\end{equation}

Furthermore, this holds even if $ n/4 < n-r < n/2$, as shown in $\cite{raven}$, and so lemma \ref{lem:raven} applies to give the result.
\par\medskip

The analysis for even $k$ will be parallel to the case $k=1$, but it will need some adjustments. First, since $\frac{j-i+1}{T(i,j)} < 1$ for each $(i,j)$, it is clear from Equation \ref{eq:eig-syt} that proving the mixing time in the case that $k=4$ will be sufficient for all even $k \ge 4$. Next, to make later calculations easier, define
$$F(T) := \frac{1}{n}\sum_{(i,i)\in T} \left(\frac{|j-i|+1}{T(i,j)}\right)^k.$$
This is larger than $\eig(T)$, since the triangle inequality gives $|j-i+1| \le |j-i| + 1$. Denote $a^T_{i,j} = \left(\frac{|j-i|+1}{T(i,j)}\right)^2$.\par
For any $T \in \text{SYT}(n)$, let $T_U := \{(i,j) \in T: j \ge i\}$ and $T_D = \{(i,j) \in T:= j < i\}$. Further, let $F_U(T) = \sum_{(i,j) \in T_U} a^T_{i,j}$ and $F_D(T)$ be defined similarly. For each $\lambda \vdash n$, let $T^\text{max}_\lambda$ be the filling of $\lambda$ which is such that $F(T) \le F(T^\text{max}_\lambda)$ for all $T \in \text{SYT}(\lambda)$; in the case more than one filling maximizes $F(T)$ then choose $T^\text{max}_\lambda$ such that $\sum_{T_U} a^T_{i,j}$ is as large as possible. Define 
$$\Lambda^+_n = \{\lambda \vdash n: F_U(T^\text{max}_\lambda) \ge F_D(T^\text{max}_\lambda)\},$$
$$\Lambda^-_n = \{\lambda \vdash n: F_U(T^\text{max}_\lambda) < F_D(T^\text{max}_\lambda)\}.$$

\par

\begin{lemma}\label{lem:useful} If $\lambda \in \Lambda^-_n$, then $\lambda' \in \Lambda_n^+$ and $F(T^\text{max}_\lambda) = F(T^\text{max}_{\lambda'})$.
\end{lemma}
\begin{proof}
Let $S = T^\text{max}_\lambda$. Notice that if $(i,j) \in S_D$, then $(j,i) \in S'_U$, simply by definition. 
$$\sum_{(i,j) \in S_U} a^T_{i,j} < \sum_{(i,j) \in S_D} a^T_{i,j} \implies \sum_{(i,j) \in S'_D} a_{i,j}^{T'} < \sum_{(i,j) \in S'_U} a_{i,j}^{T'}.$$
The statement $\lambda' \in \Lambda^+_n$ follows since $a^T_{i,j} = a_{i,j}^{T'}$. This also shows that for any $T \in \text{SYT}(n)$,
$$\sum_{(i,j) \in T} a^T_{i,j} = \sum_{(i,j) \in T'} a^{T'}_{j,i},$$
thus proving that $F(T) = F(T')$.

\end{proof}

Rewrite equation \ref{lbound} as follows;

\begin{equation*}\begin{split} \sum_{\substack{T \in \text{SYT}(n) \\ \eig(T) \neq 1}} \eig(T)^{2t}& = \sum_{\substack{\lambda \vdash n \\ \lambda \neq (n),(1^n)}}d_\lambda \sum_{T \in SYT(\lambda)} \eig(T)^{2t} + \eig(T^\text{max}_{(1^n)})^{2t}\\
&\leq  \sum_{\substack{\lambda \vdash n \\ \lambda \neq (n), (1^n)}} d_\lambda^2F(T^\text{max}_\lambda)^{2t} + \eig(T^\text{max}_{(1^n)})^{2t}\\
&= \sum_{\substack{\Lambda^+ \\ \lambda \neq (n)}} d_\lambda^2F(T^\text{max}_\lambda)^{2t} + \sum_{\substack{\Lambda^- \\ \lambda \neq (1^n)}} d_\lambda^2F(T^\text{max}_\lambda)^{2t} + \eig(T^\text{max}_{(1^n)})^{2t}.\end{split}\end{equation*}

The term $\eig(T^\text{max}_{(1^n)})$ is handled through a simple calculation,

$$\eig(T^\text{max}_{(1^n)}) = \frac{1}{n} + \frac{1}{n}\sum_{j=2}^n \left(\frac{j-2}{j}\right)^2 \le \frac{1}{n} + \frac{(n-1)(n-2)^2}{n^3} = 1 - \frac{4(n-1)^2}{n^3}.$$

Applying the inequality $1 - x \le e^{-x}$, and using $t = n\log(n) + cn$,

$$\eig(T^\text{max}_{(1^n)})^{2t} \leq \exp\left(\frac{-8n\log(n)(n-1)^2}{n^3} - \frac{4cn(n-1)^2}{n^3}\right).$$

For $n \ge 2$, \,$\frac{n(n-1)^2}{n^3} \ge \frac{(n-1)^3}{n^3} \ge \frac{1}{8}$, thus we've shown that $\eig(T^\text{max}_{(1^n)})^{2t} \le \frac{e^{-c/2}}{n}$.\par
Using lemma \ref{lem:useful} and the fact that $d_\lambda^2 = d_{\lambda'}^2$, we can bound mixing time
\begin{align}\begin{split}\label{eq:bbound}4||P_{n,k}^t - U||_{\text{T.V}} &\leq \sum_{\Lambda^+} d_\lambda^2F(T^\text{max}_\lambda) + \sum_{\Lambda^+} d_{\lambda'}^2 F(T^\text{max}_{\lambda'}) + \frac{e^{-c/2}}{n} \\
&\leq 2\sum_{\Lambda^+} d_\lambda^2 F(T^\text{max}_{\lambda}) + \frac{e^{-c/2}}{n}.\end{split}\end{align}

\begin{lemma}\label{lem:diff}  Suppose $T \in \text{SYT}(\lambda), \lambda \vdash n$. Say that $(i_1,j_1),(i_2,j_2) \in T$ are such that $T(i_1,j_1) < T(i_2,j_2)$. Let $S$ be the tableau (not necessarily standard) obtained by switching the numbers in these coordinates, then for any $k$

\begin{equation}F(S) - F(T) \begin{cases} \geq 0\,\,\,\,\text{ if } |j_2 - i_2| \ge |j_1 - i_1| \geq 0 \\ < 0\,\,\,\, \text{ if } |j_2 - i_2| < |j_1 - i_1| \end{cases}, \end{equation}

\end{lemma}
\begin{proof}
Similarly to Lemma \ref{lem:difference}, we calculate the following,
\begin{equation*}\begin{split} F(S) - F(T) &= \left(\frac{|j_2 - i_2| + 1}{T(i_1,j_1)}\right)^k + \left(\frac{|j_1 - i_1| + 1}{T(i_2,j_2)}\right)^k - \left(\frac{|j_1 - i_1| + 1}{T(i_1,j_1)}\right)^k - \left(\frac{|j_2 - i_2| + 1}{T(i_2,j_2)}\right)^k\\
&= \left(\frac{1}{T(i_1,j_1)^k} - \frac{1}{T(i_2,j_2)^k}\right)\left((|j_2 - i_2| + 1)^k - (|j_1-i_1| + 1)^k\right),\end{split}\end{equation*}
the lemma follows.
\end{proof}

The next lemma shows us we can bound $F(T)$ using a particularly simple shape.

\begin{lemma}\label{lem:maxshape} Suppose $\lambda \vdash n$ with $\lambda_1 = n - r$. If $T \in \text{SYT}(\lambda)$, then $F(T) \le F(T^\text{max}_{(n-r,1^r)})$.
\end{lemma}
\begin{proof}

Suppose we are given $T \in \text{SYT}(\lambda)$. Create a tableau $S \in \text{SYT}((n-r,1^r))$ so that $S(1,j) = T(1,j)$ for all $1 \le j \le n-r$; as in the first row of $S$ and $T$ match. The other elements of $T$ which are not in the first row are then placed down column $1$ of $S$ in increasing order, which ensures $S$ is standard. \par

\[
\begin{array}{ccc}
    \begin{ytableau}
    *(green) 1 & *(green) 3 & *(green) 4 & *(green) 6 \\
    *(yellow) 2 & *(yellow) 5 & *(yellow) 9 \\
    *(yellow) 7 & *(yellow) 8
    \end{ytableau} 
    
    & \quad \xrightarrow{\hspace{1cm}} \quad & 
    
    \begin{ytableau}
    *(green) 1 & *(green) 3 & *(green) 4 & *(green) 6 \\
    *(yellow) 2 \\
    *(yellow) 5 \\
    *(yellow) 7 \\
    *(yellow) 8 \\
    *(yellow) 9
    \end{ytableau} \\ [2.0cm]
    
    T & \quad  \quad & S
\end{array}
\]

Suppose that $(i_1,j_1) \in T$ with $i_1 \neq 1$. All the boxes in the same row and to the left of $(i_1,j_1)$, including itself, each contain numbers less than or equal to $T(i_1,j_1)$, since $T$ is standard. Similarly, all the elements in the first column above $(i_1,1)$ must also be less than or equal to $T(i_1,j_1)$. Thus, when $T(i_1,j_1)$ is placed into $S$, it must be in a box with coordinates $(I,1)$, where $I \ge i_1 + j_1 - 1$. This implies that $a^S_{I,1} \ge \frac{i_1 + j_1 - 2}{T(i_1,j_1)}$. Combining this with the general inequality

$$|j - i| \le \max\{j,i\} - 1 \le i + j - 2,$$

we have proven that $a^S_{I,1} \ge a^T_{i_1,j_1}$. Since this holds for all $(i_1,j_1) \in T, i_1 > 1$, and the first rows of $T$ and $S$ match, we have that $F(T) \le F(S)$, proving the lemma.
\end{proof}

\begin{lemma}\label{lem:maxfill} If $\lambda \vdash n$, $r \le n/2$, then $T^\text{max}_{(n-r,1^r)} = T^\rightarrow_{(n-r,1^r)}$. \end{lemma}
\begin{proof}
Suppose $T \in \text{SYT}((n-r,1^r))$ be some tableau. Create a tableau $T_1 \in \text{SYT}((n-r,1^r))$ via the following rule that $T_1(1,j) = \max\{T(1,j), T(j,1)\}, 1\le j \le r$ and $T_1(1,j) = T(1,j)$ otherwise, with the remaining column filled with the remaining elements in ascending order, as in $T_1(j,1) = \min\{T(1,j),T(j,1)\}$. The process compares $T(1,j)$ and $T(j,1)$ and either swaps them or leaves them so that the smaller number is in the row. Let us justify why this is a standard tableau. Given a $1 \le j \le r$, there are four cases: 
\begin{itemize}
\item $T(1,j) > T(j,1)$ and $T(1,j+1) > T(j+1,1)$
\item $T(1,j) > T(j,1)$ and $T(1,j+1) < T(j+1,1)$
\item $T(1,j) < T(j,1)$ and $T(1,j+1) > T(j+1,1)$
\item $T(1,j) < T(j,1)$ and $T(1,j+1) < T(j+1,1)$.
\end{itemize} 
In the first or last case, we have $T_1(1,j) < T_1(1,j+1)$ and $T_1(j,1) < T_1(j+1,1)$ by definition. In the second case, $\min\{T(1,j),T(j,1)\} = T(j,1)$ and $\min\{T(1,j+1),T(j+1,1)\} = T(1,j+1)$. Since $T(1,j) < T(1,j+1)  < T(j+1,1)$ and $T(j,1) < T(1,j) < T(1,j+1)$, we have that $T_1(j,1) < T_1(j+1,1)$ and $T_1(1,j) < T_1(1,j+1)$ by definition. Similar reasoning works out the third case. Replacing $T$ with $T_1$, we may assume from now on that $T(1,j) < T(j,1)$ for all $1 < j \le r$. \par
Assume that $T \neq T^\rightarrow_{(n-r,1^r)}$, that is there is a minimal $j_1$ such that $b = T(1,j_1) > j_1$. This forces $T(1,j_1) - 1$ to appear in a box $(i_1,1)$ with $i_1 < j_1$, since $b - 1 < T(j_1,1)$. Let $S \in \text{SYT}((n-r,1^r))$ be a tableau $S$ obtained by switching $(1,j_1)$ and $(i_1,1)$ in $T$. By lemma \ref{lem:diff}, $F(S) \ge F(T)$. We also have and $S(1,j_1) = b - 1$, and inductively, we can continue doing this to obtain a tableau $R$ with $F(R) \ge F(T)$ and $R(1,j_1) = j_1$. Since $j_1$, $T^\text{max}_{(n-r,1^r)}$ is maximized if and only if there is no such $j_1$ with $T(1,j_1) > j_1$, and thus it must be equal to $T^\rightarrow_{(n-r,1^r)}$.
\end{proof}

\begin{lemma}\label{lem:largeupper} If $T \in \text{SYT}((n-r,1^r)), r > n/2$, then $T^\text{max} = T^\downarrow_{(n-r,1^r)}$.
\end{lemma}
\begin{proof}
Since $F(T) = F(T')$, apply lemma \ref{lem:maxfill} to $T'$ so that $F(T') \le F(T^\rightarrow_{(r+1,1^{n-r-1})})$. Transposing again proves the lemma.
\end{proof}

\textbf{Partitions with large first part.} Let $\lambda \vdash n$ and set $r=n-\lambda_1 $. This section focuses on   the regime where $1 \le r \le 7/10n $ and $k \ge 4$. 

\begin{lemma} For $t = n\log(n) + cn$

$$\lim_{n\to \infty} \sum_{r=1}^{7n/10} \sum_{\substack{\lambda \dashv n \\ \lambda_1 = n-r}} d_\lambda^2 F(T^\text{max}_\lambda)^{2t} \le e^{-2c}$$

\end{lemma}

\begin{proof}
We first handle the case $r = 1$. There are exactly $n-1$ SYT of this form, each determined by the value placed in the box at $(2,1)$. It is easily seen that the maximum value of $F(T)$ over all such $T \in \text{SYT}((n-1,1))$ is $1 - \frac{1}{n} + \frac{1}{n^{k+1}}$. It suffices to show that for $t = n\log(n) + cn$, 

\begin{equation}\label{quick}(n-1)^2\left(1 - \frac{1}{n} + \frac{1}{n^{k+1}}\right)^{2t} < e^{-c}.\end{equation}

By the inequality $1 - x \le e^{-x}$, the left side of \eqref{quick} is bounded by $(n-1)^2e^{-2(\log(n)+c)(1 - \frac{1}{n^k})}$, the limit of this final term as $n$ goes to infinity is $e^{-2c}$. 

For $r > 1$, by lemmas \ref{lem:maxshape}, \ref{lem:maxfill}, and \ref{lem:largeupper}, we may bound for any $T \in \text{SYT}(\lambda)$

\begin{equation}\label{eq:bound} F(T) \le \begin{cases}\frac{n-r}{n} + \frac{1}{n}\sum_{j=1}^r \left(\frac{j+1}{n-r+j}\right)^k \text{ for } r \le n/2 \\ \frac{r+1}{n} + \frac{1}{n}\sum_{j=2}^{n-r} \left(\frac{j}{r+j}\right)^k \text{ for } r> n/2.\end{cases}\end{equation}

First, since $\frac{j+1}{n-r+j}, \frac{j}{r+j}$ is increasing with respect to $j \ge 1$, we may bound

\begin{equation}\label{eq:up1}\sum_{j=1}^r \left(\frac{j+1}{n-r+j}\right)^k \le r \left(\frac{r+1}{n}\right)^k,\end{equation}
\begin{equation}\label{eq:up2}\sum_{j=2}^{n-r} \left(\frac{j}{r+j}\right)^k \le (n-r-1)\left(\frac{n-r}{n}\right)^k.\end{equation}

Write $r = an, a \in [1/n,1/2]$. Using equations \ref{eq:bound} and \ref{eq:up1}, 

$$F(T) \le 1-a + a\left(a + \frac{1}{n}\right)^k.$$

Similarly for $a \in [1/2,7/10]$,

$$F(T) \le a + \frac{1}{n} + \left(1 - a - \frac{1}{n}\right)(1-a)^k.$$

By proposition \ref{pre:dlambda-bound} and applying Stirling's approximation,

$$\log(r!) = r\log(r) - r + \frac{1}{2}\log(2\pi r) + o(1)$$

\begin{equation}\label{eq:stirling} d_\lambda^2 \le \frac{n^{2r}}{r!} = \exp\left(an\log(n) - an\log(a) + an + \epsilon_r \right),\end{equation}

where $\forall r, \,\epsilon_r \le D\sqrt{r}$, for some constant $D$. Combining equations \eqref{eq:bbound}, \eqref{eq:bound}, and \eqref{eq:up1} with lemma \ref{lem:maxfill}, we have

\begin{align}\begin{split}\label{eq:temp0} \sum_{r = 2}^{n/2}\sum_{\substack{\lambda \vdash n \\ \lambda_1 = n-r}} d_\lambda^2 F(T^\text{max}_\lambda)^{2t} & \le \sum_{r=1}^{n/2} \frac{n^{2r}}{r!}\left(\frac{n-r}{n} + \frac{r}{n}\left(\frac{r+1}{n}\right)^k\right)^{2t} 
\end{split}\end{align}
Recall that $a = r/n$, we get
\begin{align}\begin{split}\label{eq:temp} 
\eqref{eq:temp0}&\le \sum_{r=2}^{n/2} e^{an\log(n) - an\log(a) + an + \epsilon_r + (2n\log(n) + 2cn)(-a + a(a+ \frac{1}{n})^k)} \\
&= \sum_{r=2}^{n/2} n^{n\left(a\left(a+ \frac{1}{n}\right)^k - a - \frac{a\log(a)}{\log(n)}\right)}e^{n\left(-(2c-1)a + 2ca\left(a+\frac{1}{n}\right)^k\right)}e^{\epsilon_r}. \end{split}\end{align}

Since $a + \frac{1}{n} < 1$ for all $a \in [0,1/2]$ and $n > 2, k \ge 2$,

\begin{equation}\label{eq: notimportant} a\left(a+\frac{1}{n}\right)^k \le a\left(a + \frac{1}{n}\right)^2.\end{equation}

Let $g_n(x) = x\left(x + \frac{1}{n}\right)^2 - x - \frac{x\log(x)}{\log(n)}$. Through elementary analysis, $g_n(x)$ satisfies $g_n'(x) < 0, x \in [1/n,1/2]$ and $g_n(2/n) = \frac{162}{n^3} - \frac{\log(2)}{n\log(n)} < 0$ for all large $n$. In this same range of $n,k,x$, the term $-(2c-1)a + a\left(a+\frac{1}{n}\right)^k \le -(2c-2)a$. We have shown equation \eqref{eq:temp} is bounded by

$$\sum_{r=2}^{n/2} e^{-(2c-2)r}e^{\epsilon_r} \le \sum_{r=2}^\infty D\sqrt{r}e^{-(2c-2)r}.$$ 

The sum $\sum_{r=1}^\infty re^{-pr}$ is known to be bounded by $4e^{-p}$ when $p > 1$, thus we can conclude that \eqref{eq:temp} has an upper bound by $C_1e^{-(2c-2)}$ for $c > 1$.

for some appropriate constant $A$.\par
Similarly, we use equations \eqref{eq:bbound}, \eqref{eq:bound}, and \eqref{eq:up2} with lemma \ref{lem:largeupper} to bound

\begin{align}\begin{split}\label{eq:temp12} \sum_{r=n/2}^{7n/10}\sum_{\substack{\lambda \\ \lambda_1 = n-r}} d_\lambda^2 F(T^\text{max}_\lambda)^{2t} &\le \sum_{r=n/2}^{7n/10} \frac{n^{2r}}{r!}\left(\frac{r+1}{n} + \frac{n-r-1}{n}\left(\frac{n-r}{n}\right)^k\right)^{2t} 
\end{split}\end{align}
Since $r\geq n/2$, we have
\begin{align}\begin{split}\label{eq:temp1} 
\eqref{eq:temp12}
&\leq \sum_{r=n/2}^{7n/10} e^{an\log(n) - an\log(a) + an + \epsilon_r(a) + (2n\log(n) + 2cn)\log\left(a + \frac{1}{n} + (1-a - \frac{1}{n})(1-a)^k\right)} \\
&= \sum_{r=n/2}^{7n/10} n^{n(a + 2\log\left(a + \frac{1}{n} + (1-a - \frac{1}{n})(1-a)^k\right)}e^{M_n(a)},\end{split}\end{align}

where $\lim_{n\to\infty} \max_{a \in [1/2,7/10]} \frac{|M_n(a)|}{n\log(n)} = 0$. If we set,

\begin{equation}\label{eq:q1}f_n(x) = x + 2\log\left(x + \frac{1}{n} + (1-x - \frac{1}{n})(1-x)^k\right),\end{equation}
\begin{equation}\label{eq:q2}f(x) = x + 2\log(x + (1-x)^{k+1}),\end{equation}

Then $f_n$ converges to $f$ uniformly for $x \in [1/2,7/10]$. Observe that for $k \ge 4$, $f'(x) > 0, x \in [1/2,7/10]$, thus $f$ has a global maximum at $x = 7/10$, and $c = f(7/10) < 0$. It follows that for large $n$, $f_n$ also satisfies these properties. In particular, $f_n(x) < c/2$ for all $x \in [1/2,7/10]$ and large $n$,

$$\lim_{n\to \infty} \sum_{r=n/2}^{7n/10} n^{nf_n(a)}e^{M_n(a)} \le \lim_{n \to \infty} 3n/10n^{cn/2}e^{M_n(a)} = 0.$$

By equation \eqref{eq:temp1}, this finishes the proof of the theorem
\end{proof}

\par\medskip

\textbf{Partitions with small first part.} We now focus on the regime where $7n/10 < r \le n-2$. By equation \eqref{eq:bbound}, we need only consider $T \in \text{SYT}(\lambda), \,\lambda \in \Lambda^+_n$, and thus

\begin{equation}\label{eq:smallup}F(T) \le 2F_U(T),\,\, \text{ where } F_U(T) = \sum_{(i,j) \in T_U} \left(\frac{j-i+1}{T(i,j)}\right)^k.\end{equation}

The following lemma follows immediately from the fact that $T(i,j) < S(i,j) \iff a^T_{i,j} \ge a^S_{i,j}$.

\begin{lemma}\label{lem:obvious} Suppose we are given two SYT of the same shape $S,T$, such that the ordering among the squares $T_U,S_U$ are the same. If for each $(i,j) \in T_U$ we have $T(i,j) \le S(i,j)$, then $F_U(T) \ge F_U(S)$. 
\end{lemma}

We may get upper bounds for equation \eqref{eq:smallup} in a similar manner to the odd $k$ case: first, we assume the case that $\{T(i,j): j \ge i\} = \{1,2,\ldots,N\}$, where $N$ is the cardinality of the left hand side. By lemma \ref{lem:obvious}, this is an upper bound of $F_U(S)$, for any $S$ with the same shape as $T$. \par
Second, we need to find how the numbers must be filled in $T_U$ to maximize $F_U(T)$.

\begin{lemma}\label{lem:small1}
Given a shape $\lambda$ and a fixed set of numbers to be placed in all the $(i,j)$ with $j \ge i$, we call a filling standard if the rows and columns are in decreasing order for all such squares. Then the filling $T_U^\rightarrow$ - inserting the numbers from left to right in increasing order - maximizes the value of $F_U(T)$ over standard fillings. Note that we are essentially treating $T_U$ as it's own tableau, ignoring the fact that it may make the tableau $T$ as a whole non-standard.
\end{lemma}
\begin{proof}
To see why this should maximize $F_U(T)$, we appeal to \ref{lem:diff}. Suppose that a standard filling $T_U \neq T^\rightarrow_U$, and say that $t = T^\rightarrow_U(i_1,j_1)$ is the smallest number such that $T_U(i_1,j_1) \neq T(i_1,j_1)$. Then $t-1$ must occur in a square $(i_2,j_2)$ with $j_2 - i_2 \le j_1 - i_1$ by the construction of $T_U^\rightarrow$. Thus swapping $t$ and $t-1$ creates a standard filling $T'_U$ such that $F(T_U) \le F(T'_U)$. Inductively, we have that $F(T_U) \le F(T^\rightarrow_U)$. 
\end{proof}

Third, we need to find the shape $\lambda$ which maximizes $F_U(T_U^\rightarrow)$, given that $\lambda_1 = n-r$.

\begin{lemma}\label{lem:small2} Suppose $\lambda,\mu$ are shapes such that $\lambda$ is formed by moving an inner corner $(i_1,\lambda_{i_1}), \lambda_{i_1} \ge i_1$ of $\mu$ to a higher and thus more rightward outer corner $(i_2,\lambda_{i_2}+1)$ of $\mu$. Then $F_U((T^\rightarrow_\lambda)_U) \ge F_U((T^\rightarrow_\mu)_U)$.\end{lemma}
\begin{proof} Denote $\lambda_{i_1} = j_1, \lambda_{i_2} + 1 = j_2$. Say that $T_U^\rightarrow$ and $S_U^\rightarrow$ are the maximal fillings of $\lambda_U, \mu_U$ with the numbers $\{1,2,\ldots,N\}$. Notice that for all squares $(i,j)$ below or to the right of $(i_2,j_2)$ and above or to the left of $(i_1,j_2)$, we have that $T_U^\rightarrow(i,j) = S_U^\rightarrow(i,j) + 1$. This lets us compute

\begin{align}\begin{split}\label{eq:wow}n(F_U(T^\rightarrow_U) - F_U(S^\rightarrow_U)) &= \left(\frac{j_2 - i_2 + 1}{T_U^\rightarrow(i_2,j_2)}\right)^k - \left(\frac{j_1-i_1+1}{S_U^\rightarrow(i_1,j_1)}\right)^k \\
&+ \sum_{\substack{(i,j) \in S_U^\rightarrow \cap T_U^\rightarrow \\ i_2 < i \le i_1}} \left(\frac{1}{T_U^\rightarrow(i,j)^k} - \frac{1}{S_U^\rightarrow(i,j)^k}\right)(j - i + 1)^k.\end{split}\end{align}

The largest value of $j-i+1$ is $j_2 - i_2 + 1$ over the specified range of the above sum. Since the sum is over exclusively negative terms, we use fact that $T_U^\rightarrow(i,j) = S_U^\rightarrow(i,j) + 1$ to telescope the sum to

\begin{align*}\sum_{\substack{(i,j) \in S_U^\rightarrow \cap T_U^\rightarrow \\ i_2 < i \le i_1}}& \left(\frac{1}{T_U^\rightarrow(i,j)^k} - \frac{1}{S_U^\rightarrow(i,j)^k}\right)(j - i + 1)^k\\
&\ge \left(\frac{1}{S^\rightarrow_U(i_1,j_1)^k} - \frac{1}{T^\rightarrow_U(i_2,j_2)^k}\right)(j_2 - i_2 + 1)^k.\end{align*}

Inserting this into \eqref{eq:wow}, we have that 

$$n(F_U(T^\rightarrow_U) - F_U(S_U^\rightarrow)) \ge \frac{(j_2 - i_2+1)^k - (j_1-i_1+1)^k}{S_U^\rightarrow(i_1,j_1)^k} \ge 0,$$
the inequality follows by our assumption that $(i_2,j_2)$ is above and rightward of $(i_1,j_1)$, and since all numbers are positive. We only sketched the proof here as our reasoning is completely analogous to the proof of Lemma 12 in \cite{raven}.
\end{proof}

We now have enough to finish the proof of mixing.

\begin{lemma} Over the regime where $7n/10 < r \le n - 2$, we have the following,

$$\lim_{n \to \infty} \sum_{r = 7n/10}^{n-2}\sum_{\substack{\lambda \dashv n \\ \lambda_1 = n-r}} d_\lambda^2 F(T^{\text{max}})^{2t} = 0.$$

\end{lemma}

Using lemmas \ref{lem:obvious}, \ref{lem:small1}, and \ref{lem:small2}, we see that an upper bound of $F_U(T_U)$ occurs when filling the boxes $(i,j) \in (n-r,\star)$ with the numbers $\{1,2,\ldots,N\}$, filling from left to right, where $(n-r,\star)$ is the partition which contains as many parts equal to $n-r$ as possible, and the last row contains the remainder. Denote this tableau as $\mathbf{T}_U$. so that we may complete the bound.

\begin{equation}\label{eq:up4}F_U(T) \leq F_U(\mathbf{T}_U) = \frac{n-r}{n} + \frac{1}{n}\sum_{j=1}^{n-r-1} \left(\frac{j}{n-r+j}\right)^k + \frac{1}{n}\sum_{\substack{(i,j) \in \mathbf{T}_U \\ i > 2}}a_{i,j}^{T_U^\rightarrow}.\end{equation}

For the last summand, note that the largest value of $\frac{j-i+1}{\mathbf{T}_U(i,j)}, i \ge 3$ occurs in the rightmost box of row $3$, where this term is equal to $\frac{n-r-1}{3(n-r)-1} \le \frac{1}{3}$. Also, the term $\frac{j}{n-r+j}$ is increasing. Applying these facts to equation \eqref{eq:up4}, 

\begin{equation}\label{eq:up3}\begin{split} F_U(\mathbf{T}_U) &\le \frac{n-r}{n} + \frac{n-r-1}{n}\left(\frac{n-r-1}{2n-2r-1}\right)^k + \frac{n-2(n-r)}{3^kn}\\
&=  \frac{n-r}{n}  + \frac{n-r}{2^kn}\left(1 - \frac{1}{2(2n-2r-1)}\right)^k + \frac{2r-n}{3^kn}.\end{split}\end{equation}
From equation \eqref{eq:smallup} and proposition \ref{pre:dlambda-bound},

\begin{align}\begin{split}\label{eq:one} &\sum_{r= 7n/10}^{n-2}\sum_{\substack{\lambda \vdash n\\ \lambda_1 = n-r}} d_\lambda^2F(T^\text{max}_\lambda) \\ 
&\le \sum_{r=7n/10}^{n-2} \frac{n^{2r}}{r!}\left(2\left(\frac{n-r}{n} +  \frac{n-r}{2^kn}\left(1 - \frac{1}{2(2n-2r-1)}\right)^k + \frac{2r-n}{3^kn}\right)\right)^{2t}.\end{split}\end{align}

Since $1- \frac{1}{2n-2r-1} < 1$, ignoring that term increases the term as a whole. Via the inequality, $1 - x \le e^{-x}$, and using the assumption that $k \ge 4$, we may bound the right hand side of \eqref{eq:one},

$$\sum_{r=7n/10}^{n-2} \frac{n^{2r}}{r!}\left(\frac{1361n - 1345r}{648n}\right)^{2t} \le \sum_{r=7n/10}^{n-2} \frac{n^{2r}}{r!}e^{\frac{(n\log(n)+cn)(713n-1345r)}{324n}}.$$

Once again we substitute $r = an, a \in [7/10,1]$, and use \eqref{eq:stirling} to write

$$\frac{n^{2r}}{r!}e^{\frac{(n\log(n) + cn)(713n-1345r)}{324n}} = e^{\frac{(713-1021a)}{324}n\log(n) + B_n(a)},$$

where $\frac{B_n(x)}{n} \le M, \forall n, \forall x \in [7/10,1]$ for some universal constant $M$. Notice that $f(x) = \frac{713-1021x}{324}$ is decreasing and has a maximum value at $x = 7/10$, where $f(7/10) < 0$. It follows that for large $n$, there is a constant $C >0$ such that 

$$e^{\frac{(713-1021a)}{324}n\log(n) + B_n(a)} \le e^{-Cn\log(n)}$$
$$\implies \lim_{n \to \infty} \sum_{r=7n/10}^{n-2}\sum_{\substack{\lambda \\ \lambda_1 = n-r}} d_\lambda^2 F(T^\text{max}_\lambda)^{2t} \le \lim_{n\to \infty} \frac{3n}{10}e^{-Cn\log(n)} = 0.$$\par\medskip

\textbf{Proof of \ref{thm:general-upper}(iii)} The proof for $k=2$ is largely analogous to the preceding one, but with $t = 3/2n\log (n) + cn$, and the two regimes being $1 \le r \le 3n/4$ and $3n/4 < r < n-2$. \par

First, equations \eqref{eq:temp} and \eqref{eq: notimportant} both still hold for $k \ge 2$, and thus our bounds still hold over $1 \le r \le n/2$. \par

Second, if we replace $t$ accordingly, the functions at equations \eqref{eq:q1}, \eqref{eq:q2} become 

$$f_n(x) = x + 3\log(x + \frac{1}{n} + (1-x-\frac{1}{n})(1-x)^{2}),$$
$$f(x) = x + 3\log(x + (1-x)^{3}).$$

It is clear that $f$ is increasing for $x \in [1/2,3/4]$ and $f(3/4) < 0$, showing that $t$ is a sufficient mixing time over this regime. \par

Finally, equations \eqref{eq:up4}, \eqref{eq:up3}, and \eqref{eq:one} still hold for $k \ge 2$, and so we may bound 

$$F(T^\text{max}_{\lambda}) \le \frac{41n - 37r}{2n} \implies d_\lambda^2 F(T^\text{max}_\lambda)^{2t} \le e^{n\log(n)(\frac{23-37x}{6})}e^{B_n(a)},$$

where $x = r/n$, $\frac{B_n(a)}{n} \le M, \forall n ,\,\forall x \in [3/4,1]$. Since $\frac{23-37x}{6} < 0$ on this interval, this shows our $t$ suffices for mixing.

\section{Upper bound when $k=n^\gamma$ with $\gamma\in (0,1)$}\label{sec:eig-main}

In this section, we prove \Cref{thm:general-upper}(ii). Throughout, we will assume that $k=n^\gamma$ with $\gamma\in (0,1)$ and $k$ is odd.
It suffices to bound
\begin{equation}\label{eq:tv-eig-bound}
\sum_{\substack{\lambda\vdash n\\ \lambda \neq (n)}}d_\lambda \sum_{T\in \SYT(\lambda)}\eig(T)^{2t}.\end{equation}
We split this summation into three parts based on the size of the first part of $\lambda.$ In the sections below, we show that when $t = (1-\frac{\gamma}{2})n\log n + cn,$ the following bounds hold for $c > 3$ and $n$ sufficiently large:
\begin{alignat}{2}
\sum_{r=1}^{\frac{n^\gamma}{13}} \left(\sum_{\substack{\lambda\vdash n\\ \lambda_1 = n-r}}d_\lambda \sum_{T\in \SYT(\lambda)}\eig(T)^{2t}\right) &< (e+2)e^{-c}
\quad\qquad &&\text{Large first part — \S\ref{sec:large-first-part}}
\label{eq:large-first}
\\
\sum_{r=\frac{n^\gamma}{13}}^{n} \left(\sum_{\substack{\lambda\vdash n\\ \lambda_1 = n-r}}d_\lambda \sum_{T\in \SYT(\lambda)}\eig(T)^{2t}\right)
&< 4e^{-c}
\quad\qquad &&\text{Small first part — \S\ref{sec:small-first-part}}
\label{eq:small-first}
\end{alignat}
Together, these prove
\Cref{thm:general-upper}(vi).

The following proposition, which gives bounds on $\eig(T)$ will be useful.

\begin{lemma}\label{lem:fzero}
For $\lambda \vdash n$ with $\lambda_1 = n - r$, we have
$$\eig(T^\rightarrow_{\lambda}) \leq 1 - \frac{r}{n} + O\left(\frac{1}{n^2}\right).$$

\end{lemma}
\begin{proof}
By lemmas \ref{lem:oddfill} and \ref{lem:oddshape}, we have

\begin{equation*} \eig(T) \leq \eig(T^\rightarrow_{(n-r,\star)}) \le 1 - \frac{r}{n} + \frac{1}{n}\sum_{j=1}^r \left(\frac{j-1}{n-r+j}\right)^k.\end{equation*}

Thus it suffices to show that $\lim_{n \to \infty} n\sum_{j=1}^r \left(\frac{j-1}{n-r+j}\right)^k = 0$. If we write $r = \alpha n$, then

$$n\sum_{j=1}^r \left(\frac{j-1}{n-r+j}\right)^k \leq rn\left(\frac{r-1}{n}\right)^k \le \exp(2\log(n) + (n^\gamma+1)\log(\alpha)).$$

Clearly this last term tends to $0$ as $n$ tends to infinity.
\end{proof}

\subsection{Large first part: $1\le r \le \frac{n^\gamma}{13}$}\label{sec:large-first-part}

Throughout this section, we consider $\lambda$ such that $\lambda_1 = n-r$ where $1\le r \le \frac{n^\gamma}{13}$. It will be useful to split up the tableau in $\SYT(\lambda)$ into two types:
\begin{align*}
    \SYT_1(\lambda) &:= \{T: T(2,1) \le n - r - 6rn^{1-\gamma}\}\\
    \SYT_2(\lambda) &:= \{T: T(2,1) > n - r - 6rn^{1-\gamma}\}.
\end{align*}
The next lemma shows most $T\in \SYT(\lambda)$ are in $\SYT_1(T).$
\begin{lemma}[Most tableaux are in $\SYT_1$]\label{lemma:low-v1}
For $\lambda \vdash n$ such that $\lambda_1 = n-r$ where $1\le r \le \frac{n^\gamma}{13}$,
$$\frac{|\SYT_2(\lambda)|}{d_\lambda} \le (12rn^{-\gamma})^r,$$
where we recall that $d_\lambda := \SYT(\lambda).$
\end{lemma}

\begin{proof}
Let $\mathcal{S}\subset \SYT(\lambda)$ be a set of all tableaux in $\SYT(T)$ where the elements in row two and below have some fixed ordering. We show that a randomly chosen tableau in $\mathcal{S}$ is in $\SYT_2(\lambda)$ with probability at most $(12rn^{-\gamma})^r$. Since $\SYT(\lambda)$ can be partitioned into such sets $\mathcal{S}$, this proves the result.

Note that $|\mathcal{S}| \ge \binom{n-r}{r}$ since every choice of $r$ distinct elements in $\{r+1,r+2,\cdots,n\}$ corresponds to a unique SYT in $\mathcal{S}$ with these elements below row one. By similar logic, the number of SYT in $\mathcal{S}$ with $T(2,1)> n - r - 6rn^{1-\gamma}$ is exactly $\binom{6rn^{1-\gamma}}{r}.$ The result follows, since
$$\frac{\binom{6rn^{1-\gamma}}{r}}{\binom{n-r}{r}} \le \frac{\binom{6rn^{1-\gamma}}{r}}{\binom{n/2}{r}} \le (12rn^{-\gamma})^r.$$
\end{proof}

The next lemma bounds $\eig(T)$ depending on if $T$ is in $\SYT_1(\lambda)$ or $\SYT_2(\lambda).$
\begin{lemma}[Eigenvalue bounds for $\SYT_1$ and $\SYT_2$]\label{lemma:eigs-large-first-part}
Consider $\lambda \vdash n$ such that $\lambda_1 = n-r$ where $1\le r \le \frac{n^\gamma}{13}$. Then
$$\left|\eig(T)\right|\le 
\begin{cases} 
      1 - \frac{2r}{n} + O\left(\frac{1}{n^2}\right) & \text{for }T\in \SYT_1(\lambda) \\
      1 - \frac{r}{n} + O\left(\frac{1}{n^2}\right) & \text{for }T\in \SYT_2(\lambda)
   \end{cases}.$$
\end{lemma}  

\begin{proof}
We have already proven that $|\eig(T)| = 1 - \frac{r}{n} + O\left(\frac{1}{n^2}\right)$ in general.\par
Suppose that $T \in \SYT_1(\lambda)$, and denote $s = T(2,1)$ and let $\tilde{T}^s$ be the filling of $(n-r,r)$ with $\tilde{T}^s(1,j) = j$ for $j < s$, $\tilde{T}^S(1,j) = j+1$ for $s \le j \le r$, $T(2,1) = s$, and $T(2,j) = n - r + j$. It is clear that $|\eig(T)| < |\eig(\tilde{T}^s)|$ and that
\begin{align*}\eig(\tilde{T}^s) &= \frac{s-1}{n} + \frac{1}{n}\sum_{i = s+1}^{n-r} \left(\frac{i-1}{i}\right)^k + \frac{1}{n}\sum_{j=2}^r \left(\frac{j-1}{n-r+j}\right)^k \\
&\leq \frac{s-1}{n} + \frac{n - r - s+1}{n}\left(\frac{n-1}{n}\right)^k + 2^{-k}\\
&= 1 - \frac{r}{n} - \frac{n-r-s+1}{n}\cdot\frac{n^k - (n-1)^k}{n^k} + 2^{-k}.\end{align*}
Finally observe,
\begin{align*}
    \frac{n^{k} - (n-1)^{k}}{n^{k}} &\ge \frac{k}{n}\cdot \frac{(n-1)^{k-1}}{n^{k-1}}= \frac{k}{n}\left(1 - \frac{1}{n}\right)^{k-1} \ge \frac{k}{2n}.
\end{align*}
By assumption, we have $s \le n - r - 6rn^{1-\gamma}$, and so
$$\eig(T) \leq 1 - \frac{r}{n} - \frac{6rn^{1-\gamma}n^{\gamma-1}}{2n} + 2^{-k} \leq 1 - \frac{2r}{n},$$
proving out claim.
\end{proof}

Together, the above two lemmas show that---in essence---almost all eigenvalues are from $\SYT_1(\lambda)$ and are very small, while a small number of eigenvalues are from $\SYT_2(\lambda)$ and are somewhat larger.

\begin{proof}[Proof of \eqref{eq:large-first}.] We will split the summation in \eqref{eq:large-first} depending on if $T$ is in $\SYT_1(\lambda)$ or $\SYT_2(\lambda)$. Beginning with the first case,
\begin{align}
    \sum_{r=1}^{\frac{n^\gamma}{13}} \left(\sum_{\substack{\lambda\vdash n\\ \lambda_1 = n-r}}d_\lambda \sum_{T\in \SYT_1(\lambda)}\eig(T)^{2t}\right)
    &\le \sum_{r=1}^\frac{n^\gamma}{13} \left(1 - \frac{2r}{n} + O\left(\frac{1}{n^2}\right)\right)^{2t} \sum_{\substack{\lambda\vdash n\\ \lambda_1 = n-r}} |\SYT_1(\lambda)|d_\lambda \label{eq:cat1}\\
    &\le \sum_{r=1}^\frac{n^\gamma}{13} \left(1 - \frac{2r}{n} + O\left(\frac{1}{n^2}\right)\right)^{2t} \left(\sum_{\substack{\lambda\vdash n\\ \lambda_1 = n-r}} d_\lambda^2\right) \label{eq:cat2}\\
    &\le \sum_{r=1}^\frac{n^\gamma}{13} \frac{n^{2r}}{r!}\left(1 - \frac{2r}{n} + O\left(\frac{1}{n^2}\right)\right)^{2t} \label{eq:cat3},
\end{align}
where \eqref{eq:cat1} follows from \Cref{lemma:eigs-large-first-part}, \eqref{eq:cat2} follows from the bound $|\SYT_1(\lambda)|\le |\SYT(\lambda)| = d_\lambda$, and \eqref{eq:cat3} follows from \Cref{pre:dlambda-bound}. Taking $t=\frac{1}{2}n\log n + cn,$ this is at most
\begin{align}
    \sum_{r=1}^\frac{n^\gamma}{13} \frac{n^{2r}}{r!} e^{-(n\log n + 2cn)\left(\frac{2r}{n}-\bigO{n^2}\right)}
    &= \sum_{r=1}^\frac{n^\gamma}{13} \frac{n^{2r}}{r!} n^{-2r + O\left(\frac{1}{n}\right)} e^{-2rc + c \bigO{n}} \nonumber \\
    &= \sum_{r=1}^\frac{n^\gamma}{13} e^{-2rc + c \bigO{n}} \frac{n^{\bigO{n}}}{r!} \nonumber \\
    &< \sum_{r=1}^\frac{n^\gamma}{13} \frac{2e^{-c}}{r!} \label{eq:dog1}\\
    &< 2(e-1)e^{-c}, \label{eq:dog2}
\end{align}
where \eqref{eq:dog1} follows from the crude observations that for $n$ sufficiently large, $n^{\bigO{n}} < 2$ and $e^{-2rc + c\bigO{n}} < e^{-c}$. Equation \eqref{eq:dog2} follows from the Taylor expansion of $e^x.$

In the second case, where $T\in \SYT_2(\lambda)$, we have
\begin{align}
    \sum_{r=1}^{\frac{n^\gamma}{13}} \left(\sum_{\substack{\lambda\vdash n\\ \lambda_1 = n-r}}d_\lambda \sum_{T\in \SYT_2(\lambda)}\eig(T)^{2t}\right)
    &\le \sum_{r=1}^\frac{n^\gamma}{13} \left(1 - \frac{r}{n} + \bigO{n^2}\right)^{2t} \sum_{\substack{\lambda\vdash n\\ \lambda_1 = n-r}} |\SYT_2|d_\lambda \label{eq:lizard1}\\
    &\le \sum_{r=1}^\frac{n^\gamma}{13} \left(1 - \frac{r}{n} + \bigO{n^2}\right)^{2t} \left(\sum_{\substack{\lambda\vdash n\\ \lambda_1 = n-r}} (12rn^{-\gamma})^r d_\lambda^2\right) \label{eq:lizard2}\\
    &\le \sum_{r=1}^\frac{n^\gamma}{13} \frac{n^{2r}}{r!}\cdot (12rn^{-\gamma})^r\cdot \left(1 - \frac{r}{n} + \bigO{n^2}\right)^{2t}, \label{eq:lizard3}
\end{align}
where \eqref{eq:lizard1} follows from \Cref{lemma:eigs-large-first-part}, \eqref{eq:lizard2} follows from \Cref{lemma:low-v1}, and \eqref{eq:lizard3} follows from \Cref{pre:dlambda-bound}. Using the bound $r! \ge r^me^{1-r}$, this is at most
\begin{align*}
    \sum_{r=1}^\frac{n^\gamma}{13} (12rn^{-\gamma})^r n^{2r} r^{-r} e^{r-1} \left(1 - \frac{r}{n} + \bigO{n^2}\right)^{2t}.
\end{align*}
Taking $t=(1 - \frac{\gamma}{2})n\log n + cn$, this is at most
\begin{align*}
    &\sum_{r=1}^\frac{n^\gamma}{13} (12rn^{-\gamma})^r n^{2r} r^{-r} e^{r-1} e^{-((2-\gamma)n\log n + 2cn)\left(\frac{r}{n}+\bigO{n^2}\right)}
    \\ &= \sum_{r=1}^\frac{n^\gamma}{13} 12^r e^{-2rc + r - 1 + c\bigO{n}} n^{\bigO{n}}\\
    &< \,2\sum_{r=1}^\frac{n^\gamma}{13} e^{-2rc + (1 + \log 12) r - 1 + c\bigO{n}},
\end{align*}
where in the last line we used that $n^{\bigO{n}} < 2$ for $n$ sufficiently large. Now, whenever $c>3$ and $n$ is sufficiently large, the remaining expression is less than
\begin{align*}
2\sum_{r=1}^\frac{n^\gamma}{13} e^{-rc} < 2\sum_{r=1}^\infty e^{-rc} < 4e^{-c}.
\end{align*}
Combining the two cases completes our calculation.
\end{proof}

\subsection{Small first part ($\frac{n^\gamma}{13} \le r \le n$)}\label{sec:small-first-part}
We proceed to the main calculation. Reminder to the reader that it is only necessary to consider $\lambda \in \Lambda^+$ by lemma \ref{lem:biglam}. \par

\begin{proof}[Proof of \eqref{eq:small-first}]
Applying \Cref{lem:fzero}, we find
\begin{align}
    \sum_{r=\frac{n^\gamma}{13}}^{n-1} \left(\sum_{\substack{\lambda: \eig(T_\lambda^\rightarrow)\\ \lambda_1 = n-r}}d_\lambda^2 \eig(T^\rightarrow_\lambda)^{2t}\right)
    &\le \sum_{r=\frac{n^\gamma}{13}}^{n} \frac{n^{2r}}{r!}\left(1 - \frac{r}{n} + \bigO{n^2}\right)^{2t} \label{eq:rhino1} \\
    &\le \sum_{r=\frac{n^\gamma}{13}}^{n} n^{2r}r^{-r}e^{r-1} \left(1 - \frac{r}{n} + \bigO{n^2}\right)^{2t},\label{eq:rhino2}
\end{align}
where \eqref{eq:rhino1} follows from \Cref{pre:dlambda-bound} and \eqref{eq:rhino2} follows from the bound $r! \ge r^re^{1-r}$. Continuing, and taking $t = (1 - \frac{\gamma}{2})n\log n + cn,$ this is less than

\begin{align}
    \sum_{r=\frac{n^\gamma}{13}}^{n} n^{2r}r^{-r}e^{r-1} e^{-((2-\gamma)n\log n + 2cn)\left(\frac{r}{n} + \bigO{n^2}\right)}
    &= \sum_{r=\frac{n^\gamma}{13}}^{n} e^{-2rc+r-1+c\bigO{n}}r^{-r}n^{\gamma r + \bigO{n}}\\
    &\le \sum_{r=\frac{n^\gamma}{13}}^{n} e^{-2rc+r-1+c\bigO{n}} 13^r n^{\bigO{n}},\label{eq:potato1}
\end{align}
where \eqref{eq:potato1} follows because $r^{-r}n^{\gamma r} = (\frac{n^\gamma}{r})^r \le 13^r$. Noting that $n^{\bigO{n}} < 2$ for $n$ sufficiently large, whenever $c > 3$, this is less than
\begin{align*}
    \sum_{r=\frac{n^\gamma}{13}}^{n}  2e^{-2rc+(1 + \log 13)r-1+c\bigO{n}} < 2 \sum_{r=\frac{n^\gamma}{13}}^{n} e^{-rc} < 2\sum_{r=1}^\infty (e^{-rc}) < 4e^{-c}.
\end{align*}
\end{proof}

\section{An $\ell_2$ lower bound}\label{sec:l2}

In the previous two sections we gave upper bounds for the mixing time of $\OSC_{n,k}$ by using \Cref{prop:tv-eig-bound}. In this section, we show that these are the best bounds achievable by this technique.

We consider the contribution of eigenvalues corresponding to tableau of shape $\lambda = (n-1,1).$ In this section, we will let $T_i$ denote the standard Young tableau of shape $(n-1,1)$ such that $T_i(2,1)=i.$ We can find an exact formula for $\eig(T_i)$, and by considering only these eigenvalues, we produce a lower bound on the $\ell^2$ distance via 
\begin{equation}\label{eq:l2-lower}
\norm{\frac{\OSC_{n,k}^t}{U}- 1}_2^2 = \sum_{\substack{\lambda\vdash n\\ \lambda \neq (n)}}d_\lambda \sum_{\lambda\in SYT(\lambda)}\eig(T)^{2t} \ge (n-1) \sum_{i=2}^n \eig(T_i),
\end{equation}
where we used that $d_{(n-1,1)} = n-1.$

We can write the eigenvalue corresponding to $T_i$ in a simple way.
\begin{proposition}\label{prop:ti}
For $T_i$ defined above, we have
$$\eig(T_i) = \frac{1}{n}\left(i-1 + \sum_{j=i+1}^n \left(\frac{j-1}{j}\right)^k\right).$$
\end{proposition}

\begin{proof}
By definition, if $j < i$, then $T(1,j) = j$, and if $j \ge i$, then $T(1,j) = j+1$. Using (Theorem 7), we calculate
$$\eig(T_i) = \sum_{1 \le j < i} \left(\frac{j}{T(1,j)}\right)^k + \sum_{n \ge j \ge i} \left(\frac{j}{T(1,j)}\right)^k = \frac{1}{n}\left(i - 1 + \sum_{j=i+1}^n \left(\frac{j-1}{j}\right)^k\right).$$
\end{proof}

\begin{corollary}[A large universal eigenvalue]\label{cor:large-eig} 
For $i=n$, Proposition \ref{prop:ti} gives
$$\eig(T_n) = 1 - \frac{1}{n}$$
\end{corollary}

We may now prove \Cref{thm:l2-lower-bounds}.
\begin{proof}[Proof of \Cref{thm:l2-lower-bounds}] 
Using \Cref{eq:l2-lower} and \Cref{cor:large-eig}, for all $k\ge 1$ and $t = \frac{1}{2}n\log n - cn,$
\begin{align*}
    \sum_{\substack{\lambda\vdash n\\ \lambda \neq (n)}}d_\lambda \sum_{\lambda\in SYT(\lambda)}\eig(T)^{2t} &> (n-1) \eig(T_n)^{2t}\\
    &> (n-1)\left(1 - \new{\frac{1}{n}}\right)^{n\log n - 2cn} > \new{\frac{1}{4}}e^{2c},
\end{align*}
where the last bound is loose (in fact, $1/4$ can be replaced with any constant less than $1$).
Now consider the case when $k = n^\gamma$. First observe that
\begin{align*}
    \eig(T_i) &= \frac{1}{n}\left(i-1 + \sum_{j=i+1}^n \left(\frac{j-1}{j}\right)^k\right)\\
    &\ge 1 - \frac{1}{n}\left(\frac{i}{i-1} + \sum_{j=i}^{n-1} \frac{k}{j+1}\right),
\end{align*}
\new{where we used that $(j+1)^k - j^k \le k(j+1)^{k-1}$, which follows from the convexity of $x^k$ for $k\ge 1, x\ge 0.$}
When $i\ge n-\frac{n^{1-\gamma}}{\log n}$, $\frac{i}{i-1} = 1 + o(1),$ and furthermore,
\begin{align*}
    \sum_{j=i}^{n-1} \frac{k}{j+1}
    \le \frac{n^{1-\gamma}}{\log n}\left(\frac{n^\gamma}{n-\frac{n^{1-\gamma}}{\log n}}\right) = o(1).
\end{align*}
Proceeding from \eqref{eq:l2-lower}, and taking $t = (1 - \frac{\gamma}{2})n\log n - \frac{1}{2}n\log \log n - cn,$
\begin{align*}
    (n-1) \sum_{i=2}^n \eig(T_i) &> (n-1) \sum_{i=n-\frac{n^{1-\gamma}}{\log n}}^n \eig(T_i)^{2t}\\
    &\ge (n-1)\frac{n^{1-\gamma}}{\log n}\left(1 - \frac{1}{n} + o\left(\frac{1}{n}\right)\right)^{2t} > e^{2c}.
\end{align*}
\end{proof}

\section{Lower bound}\label{sec:lb}

We devote this section to proving \Cref{thm:lower}, which gives a lower bound for the mixing time of $\OSC_{n,k}$ using a coupon-collecting argument.
Recall that one step of the one-sided $k-$transposition shuffle involves first selecting a card $r_i$ uniformly from $\{1,2,\cdots,n\}$ and then selecting a set of cards $L_i$ by sampling $k$ times uniformly from $\{1,2,\cdots,r_i\}$ with replacement (thus, $|L_i|\le k$). Following the strategy in \cite{raven} closely, we observe that intuitively this process is relatively unlikely to choose cards near the top of the deck. Therefore, we focus on $$V_n = \{n - n/m + 1, \cdots, n-1, n\},$$ the set representing the top $n/m$ cards in the deck, where $m = k\log n$. Further let
$$B_n := \{\rho\in S_n: \rho\text{ has at least 1 fixed point in }V_n\}.$$
The idea is to use the simple bound
\begin{equation}
    \norm{\OSC_{k,n}^t - U}_{\TV} \ge \OSC_{n,k}^t(B_n) - U(B_n).
\end{equation}
It is easy to see that $U(B_n) \le 1/m \rightarrow 0$ as $n\rightarrow \infty$. Thus, it suffices to bound $\OSC_{n,k}^t(B_n)$ appropriately. Let $U_n^t$ be the set of untouched cards in $V_n$ after $t$ iterations of the shuffle. Then
$$\OSC_{k,n}^t(B_n) \ge \PP(|U_n^t|\ge 1)$$
(i.e., the probability that there exists an untouched card in $V_n$). Thus, we have reduced our problem to a variant of coupon collecting.
We begin by modeling what happens in one iteration of the shuffle:
\begin{align}
|V_n\setminus U_n^{t+1}| - |V_n\setminus U_n^t| &= |(\{r^{t+1}\} \cup L^{t+1})\cap U_n^t| \nonumber \\
&\le \mathds{1}[r^{t+1}\in U_n^t] + |L^{t+1}\cap U_n^t| \nonumber \\
&\le \mathds{1}[r^{t+1}\in U_n^t, |L^{t+1}\cap U_n^t| > 0] + \mathds{1}[r^{t+1}\in U_n^t, |L^{t+1}\cap U_n^t| = 0] + |L^{t+1}\cap U_n^t| \nonumber \\
&\le \mathds{1}[|L^{t+1}\cap U_n^t| > 0] + |L^{t+1}\cap U_n^t| + \mathds{1}[r^{t+1}\in U_n^t, |L^{t+1}\cap U_n^t| = 0].\label{eq:change}
\end{align}
Clearly,
\begin{equation}\label{eq:right-prob}
    \PP(r^{t+1}\in U_n^t, |L^{t+1}\cap U_n^t| = 0) \le \PP(r^{t+1}\in U_n^t) = \frac{|U_n^t|}{n}.
\end{equation}
We now focus our attention on the quantity $|L^{t+1}\cap U_n^t|$.
\begin{proposition}\label{prop:left-prob}
Let $s\in [n]$. We have 
$$\PP(|L^{t+1}\cap U_n^t|\ge s) \le  \frac{1}{m}\left(\frac{k|U_n^t|}{n-\frac{n}{m}}\right)^{s}.$$
\end{proposition}
\begin{proof}
If $r_t \le n - \frac{n}{m},$ then $|L^{t+1}\cap U_n^t| = 0.$ Otherwise, we can use the naïve bound $$\PP(l\in U_n^t| l\in L^{t+1}, r^t > n - \frac{n}{m}) \le \frac{|U_n^t|}{n-\frac{n}{m}},$$ which arises by considering the extreme case when $r^t = n$. This gives 
\begin{align*}
    \PP(|L^{t+1}\cap U_n^t|\ge s)
    &\le \PP(r_t > n - \frac{n}{m}) \cdot \PP\left[\text{Bin}\left(k, \frac{|U_n^t|}{n-\frac{n}{m}}\right)\ge s\right]\\
    &\le \frac{1}{m}\left(\frac{k|U_n^t|}{n-\frac{n}{m}}\right)^{s}.
\end{align*}
\end{proof}
We now let $\TT_i$ be the number of integers $t\ge 0$ for which $|V_n \setminus U_n^t|=i$. Thus, setting $\TT := \TT_0 + \TT_1 + \cdots + \TT_{\frac{n}{m}-1}$ to be the time it takes to collect all cards in $V_n$,
\begin{equation}\label{eq:process}
    \PP(|U_n^t| \ge 1) = \PP(\TT > t).
\end{equation}
Let $S_{i,j}$ be the event that we go from having exactly $j$ collected cards straight to having more than $i$ collected cards. Then
\begin{align*}
    \PP(S_{i,j}) \le \frac{\PP(|L^{t+1}\cap U_n^t|\ge i-j)}{\PP(|U_n^{t+1}| < |U_n^{t}|)} \le \frac{ \frac{1}{m}\left(\frac{k|U_n^t|}{n-n/m}\right)^{i-j}}{|U_n^t|/n},
\end{align*}
where we applied \Cref{prop:left-prob} and bounded the denominator from below by $\PP(r^{t+1}\in U_n^t) = \frac{|U_n^t|}{n}$. Thus,
\begin{equation*}
    \PP(\TT_i = 0) \le \sum_{j=i-k}^{i-1}\PP(S_{i,j}) < \sum_{s=1}^k \frac{ \frac{1}{m}\left(\frac{k|U_n^t|}{n-n/m}\right)^{s}}{|U_n^t|/n} \new{= \frac{k}{m-1} \sum_{s=1}^k \left(\frac{k|U_n^t|}{n-\frac{n}{m}}\right)^{s-1}} \nonumber < \frac{2k}{m},
\end{equation*}
where the last inequality follows easily \new{because for $n$ large, $k = \frac{m}{\log n}\le m-1$} and $\frac{k|U_n^t|}{n-n/m}\le \frac{kn/m}{n-n/m} = \frac{1}{(1-1/m)\log n} < \frac{1}{3}$.
Then $\TT_i$ stochastically dominates $\TT_i'$ given by
\begin{align*}
    \PP(\TT_i'=j) = 
    \begin{cases} 
      2k/m ,& j=0 \\
      (1-2k/m)p_i(1-p_i)^{j-1} ,& j\ge 1,
   \end{cases}
\end{align*}
where
\begin{align}
p_i &:= \PP(|U_n^{t+1}| < |U_n^t| : |U_n^t| = \frac{n}{m}-i) \nonumber\\
&\le \PP(|L^{t+1}\cap U_n^t|\ge 1) + \PP(r^{t+1}\in U_n^t) \label{eq:car1}\\
&\le \frac{1}{m}\left(\frac{k(\frac{n}{m}-i)}{n-\frac{n}{m}}\right) + \frac{\frac{n}{m}-i}{n} = \frac{k + m - 1}{(m-1)n}\left(\frac{n}{m} - i\right),\label{eq:car2}
\end{align}
where \eqref{eq:car1} is a consequence of \eqref{eq:change}, and \eqref{eq:car2} follows by applying \Cref{prop:left-prob} and \eqref{eq:right-prob} and then simplifying.

We may now apply Chebyshev's inequality to $\TT' = \TT_0' + \TT_1' + \cdots + \TT_{\frac{n}{m} - 1}'$. Recalling that $m = k\log n,$
\begin{align*}
    \EE[\TT'] = \sum_{i=0}^{\frac{n}{m}-1} \frac{m-2k}{mp_i} &\ge \frac{m-2k}{m}\cdot \frac{m-1}{k+m-1}\cdot n \sum_{i=0}^{\frac{n}{m}-1} \frac{1}{\frac{n}{m} - i}\\
    &\ge \left(1 - \frac{4}{\log n}\right)\cdot \left( n\log \left(\frac{n}{k}\right) - n\log \log n\right) \\ 
    &\ge n\log \left(\frac{n}{k}\right) - n\log \log n - 4n
\end{align*}
and
\begin{equation*}
    \Var[\TT'] \le \sum_{i=0}^{\frac{n}{m}-1} \frac{1}{p_i^2} \le \sum_{i=1}^{\frac{n}{m}} \frac{n^2}{i^2} \le \frac{\pi^2}{6}n^2,
\end{equation*}
where we use that $p_i\ge \frac{n/m - i}{n}.$ Thus, for $t = n\log \left(\frac{n}{k}\right) - n\log \log n - cn,$
$$\PP(\TT' \le t) \le \frac{\pi^2}{6(c-4)^2}.$$
This implies, using \Cref{eq:process}, that
\begin{equation*}
    \PP(|U_n^t| \ge 1) = \PP(\TT > t) \ge 1 - \frac{\pi^2}{6(c-4)^2},
\end{equation*}
completing the proof of \Cref{thm:lower}.

\section{The case where $k$ is big}\label{sec:lk}

In this section, we prove \Cref{thm:large-k}. In particular, if $k = \Omega(n \log n)$, then $ \OSC_{n, k}$ mixes in $O(n)$ steps without cutoff. The proof is simple. We will show that the mixing time is at most order $n$ using a coupling argument.

Let $t_{\textup{rel}}= (1-\beta)^{-1}$, where $\beta$ is the second largest eigenvalue of $\OSC_{n,k}$. Then \Cref{cor:large-eig} gives that $t_{\textup{rel}} = n$. Thus, as a consequence,
$$\lim_{n \rightarrow \infty} \frac{t_{\textup{mix}}(\varepsilon)}{t_{\textup{rel}}} \neq \infty,$$
which implies that in this regime, there is no cutoff (see, e.g., Proposition 18.4 of \cite{Levin2008MarkovCA}).

\begin{proof}[Proof of Theorem \ref{thm:large-k}]
Let $Q$ be the transition matrix of star transpositions. \new{More precisely, define $Q$ such that $Q(x,xs)= \frac{1}{n} $ for every $s\in \{(n \quad i), i \in [n]\}$ and $ x \in S_n$ and $Q(x,y)=0$ otherwise.} Diaconis \cite{PD} proved that there is a universal, positive constant $A$ such that 
\begin{equation*}
\max_{x \in S_n} \Vert Q^t_x- U \Vert_{\TV} \leq Ae^{-c},
\end{equation*}
where $t=n \log n + c n$  with $c>1$.
Combining this with the fact that there is a coupling time $\tau$ such that
\[
\Vert Q^t_x- U \Vert_{\TV} = \Pr(\tau>t)
\]
for every $x \in S_n$ (see Proposition 4.7 of \cite{Levin2008MarkovCA} for a reference), we get that \begin{equation}\label{star}
\Pr(\tau>n \log n +c n) <Ae^{-c},
\end{equation}
if $c>1$.
We define the following coupling time $T$ for the one-sided $k$-transposition shuffle. 

Let $X_t$ and $Y_t$ be two copies of the one-sided $k$-transposition shuffle. We recall that one step of the one-sided transposition shuffle consists of choosing $j\in \{1,\cdots,n\}$, choosing $i_1,\cdots,i_k\in \{1,\cdots,j\},$ and then applying the permutation $(j; i_1,\ldots, i_{k})$. Then a coupling is given as follows: For $1\leq j \leq n-1$ we apply the same permutation $(j; i_1,\ldots, i_{k})$ to both chains. That is
\[X_{t+1}= X_t (j; i_1,\ldots, i_{k}) \mbox{ and }Y_{t+1}= Y_t (j; i_1,\ldots, i_{k}).\]
When $j=n,$ the permutation applied is equivalent to $k$ star transpositions. In this case, we couple $X_t$ and $Y_t$ according to the star transpositions coupling. 

Let $T$ be the first time that $X_t=Y_t$. The standard coupling inequality says
\[d^{(n,k)}(t) \leq \PP \left(T>t\right).\]
Roughly speaking, the coupling progresses whenever $j=n$, so it suffices to consider how often this happens. Let $B$ be a Binomial$(t, 1/n)$ random variable counting the number of times before time $t+1$ that $j=n$. Then,
\[\PP \left(T>t\right)\leq  \Pr\left(T>t , B> \frac{t}{2n}\right) + \Pr\left(B\leq \frac{t}{2n}\right).\]
\new{We consider the case $k = \Omega(n \log n)$. Note that every time that $j=n$, we perform $k$ star transpositions. Hence, for $t=4d n,$ (with $d>\frac{n\log n}{k}$), we have that the event $B>\frac{t}{2n}$ implies that we have applied at least $2dk$ star transpositions. Therefore, 
$$\Pr\left(T>t , B> \frac{t}{2n}\right) \leq \Pr(\tau>2 d  k).$$
Thus, 
\begin{equation}\label{couple}
\Pr(T>t) \leq \Pr(\tau>2 d k )+ e^{-2d} \leq \Pr(\tau>2 n \log n)+ e^{-2d}, 
\end{equation}
where we bounded the tail of the binomial distribution.} Combining equations \eqref{star} and \eqref{couple}, we get that there are positive constants $A,B$ that are universal on $n$ such that 
\[d^{(n,k)}(t) \leq Ae^{-Bd},\]
where $t=4 d  n$. For $k \in [n,n \log n] $ the same argument holds for $t= O \left(\frac{n^2 \log n}{k}\right)$.

We now present the lower bound. Equation (12.15) of \cite{Levin2008MarkovCA} says that for every eigenvalue $\beta\neq 1$ of $P_{n,k}$ we have that
\[|\beta|^t \leq 2 \norm{\OSC_{n,k}^t - U}_{\TV}. \]
Using \Cref{cor:large-eig} we have that 
\[\left(1- \frac{1}{n-1} \right)^t \leq |\eig(T_n)|^t \leq 2 \norm{\OSC_{n,k}^t - U}_{\TV},\]
which shows that $t_{\text{mix}}(\varepsilon)= \Omega(n)$ for every $k$.

\end{proof}

\section*{Appendix: Lifting Eigenvectors}

In this section, we will prove \Cref{thm:eig-syt}, which gives the eigenvalues of $P_{n,k}$. Our approach closely follows that of \cite{raven}, who in turn closely follows the approach of \cite{dieker}. In brief, the strategy is to recursively find the eigenvectors of $\OSC_{n+1,k}$ in terms of those of $\OSC_{n,k}$ by considering the group algebra $\SS_n = \CC[S_n]$ and its representations. We now introduce some background, following the outline of \cite{raven} closely.

Let $[n] = \{1,2,\cdots,n\}$ for $n\in \mathbb{N}.$ Given $n\in \mathbb{N},$ allow $W^n$ to be the set of words $w = w_1\cdot w_2\cdot \ldots \cdot w_n$ of length $n$ with elements $w_i\in [n].$ We let $S_n$ act on $W^n$ via place permutations, i.e., for $\sigma\in S_n,$ $\sigma(w_1\cdot w_2\cdot \ldots \cdot w_n) := w_{\sigma^{-1}(1)}\cdot w_{\sigma^{-1}(2)}\cdot \ldots \cdot w_{\sigma^{-1}(n)}.$ Now let $M^n$ be the vector space over $\CC$ with basis $W^n,$ on which the $S^n-$action we define above extends to.

For $w\in W^n$, let $\eval_i(w)$ be the number of occurrences of $i$ in the word $w$. Then define $\eval(w) := (\eval_1(w),\cdots,\eval_n(w))$ be the \textit{evaluation} of $w$. If $\eval(w)$ is non-increasing, then we identify $\eval(w)$ with a partition $\lambda\vdash n$ where $\lambda_1 = \eval_1(w), \lambda_2 = \eval_2(w), \cdots.$ Furthermore, to any standard Young tableau $T$ of shape $\lambda\vdash n$ we may associate a word $w = w_1\cdot w_2\cdot \ldots \cdot w_n \in W^n,$ where $w_{T(i,j)} = i$ for all boxes $(i,j)$ in $T$. 
\new{There is at most one standard Young tableau associated with each word.}

\begin{definition}
Given $\lambda\vdash n,$ we can associate to it a simple module $S^\lambda$ of $\SS_n$ called the Specht module for $\lambda$. It has dimension $d_\lambda := |\SYT(T)|$.
\end{definition}

\begin{definition}
Given $\lambda\vdash n$, define $M^\lambda$ to be the span of $\{w\in W^n: \eval(w) = \lambda\}.$ This is clearly a $\SS_n$-submodule of $M^n$.
\end{definition}

We are now ready to see how this relates to card shuffles. Let $(1^n):=(1,\cdots,1)$ denote the partition of all ones. Then $M^{(1^n)}$ is spanned by the $n!$ permutations of the word $1\cdot 2\cdot \ldots \cdot n\in W^n$; thus, card shuffles can be studied as linear operators on $M^{(1^n)}$.

Indeed, consider the one-sided $k-$transposition shuffle on $n$ cards as the following element of the group algebra $\SS_n$:
$$\sum_{1\le j\le n}\sum_{1\le i_1,\cdots,i_k\le j} P_{n,k}((j;i_1,\cdots,i_k)) (j;i_1,\cdots,i_k) = \sum_{1\le j\le n}\sum_{1\le i_1,\cdots,i_k\le j} \frac{1}{nj^k} (j;i_1,\cdots,i_k).$$
To simplify our calculations, we scale this operator by $n$ to get the operator
$$Q_{n,k} := \sum_{1\le j\le n}\sum_{1\le i_1,\cdots,i_k\le j} \frac{1}{j^k}(j;i_1,\cdots,i_k).$$
We seek to determine the eigenvalues of $Q_{n,k}$ on $M^{(1^n)}.$ The following standard results indicate that it suffices to find the eigenvalues of $Q_{n,k}$ on $S^\lambda$, and that we may study the action of $Q_{n,k}$ on $S^\lambda$ within the module $M^\lambda.$
\begin{lemma}
Given $\lambda\vdash n,$
$$M^\lambda \cong \bigoplus_{\mu \trianglerighteq \lambda} K_{\lambda,\mu} S^\mu,$$ where $K_{\lambda,\lambda} = 1$ and $K_{\lambda,\mu}$ are the Kostka numbers. In particular,
$$M^{(1^n)} \cong \bigoplus_{\lambda\vdash n} d_\lambda S^{\lambda}.$$
Recall that for partitions $\lambda,\mu\vdash n,$ we write $\lambda\trianglerighteq \mu$ if $\mu$ can be obtained by moving boxes in $\lambda$ down and to the left.
\end{lemma}

The following key operators will allow us to connect $Q_{n+1,k}$ with $Q_{n,k}.$

\begin{definition}
For $a\in[n+1],$ define the \textbf{adding operator} $\Phi_a:M^n\rightarrow M^{n+1}$ so that for $w\in W^n$,
$$\Phi_a(w) := w\cdot a.$$
In other words, $\Phi_a$ appends the symbol $a$ to the end of the word $w$.

Furthermore, for $a,b\in [n],$ define the \textbf{switching operator} $\Theta_{b,a}:M^n\rightarrow M^n$ so that for $w = w_1\cdot w_2\cdot \ldots \cdot w_n\in W^n,$
$$\Theta_{b,a}(w) := \sum_{\substack{1\le i\le n\\ w_i=b}} w_1\cdot \ldots \cdot w_{i-1}\cdot a\cdot w_{i+1}\cdot \ldots \cdot w_n.$$
In other words, $\Theta_{b,a}$ sums all words formed by replacing an occurrence of the symbol $a$ in $w$ with the symbol $b$.
\end{definition}

The operators defined above behave nicely when restricted to the modules $M^{\lambda}.$ The following definition will be useful in this case.
\begin{definition}
Given an $n-$tuple $\lambda=(\lambda_1,\cdots,\lambda_n)$ of non-negative integers summing to $n$, we define $\lambda + e_a$ to be the $(n+1)-$tuple $(\lambda_1,\cdots,\lambda_n,0) + \underbrace{(0,\cdots,0,1)}_{a}$ of non-negative integers summing to $n+1.$
\end{definition}

\begin{lemma}[Lemma 31 in \cite{raven}]
For $a\in [n+1]$ and an $n-$tuple $\lambda$ of non-negative integers summing to $n$, $$\Phi_a:M^{\lambda}\rightarrow M^{\lambda+e_a},$$
and that for $a,b\in [n],$ and $n-$tuples $\lambda,\mu$ of non-negative integers summing to $n$ where $\lambda+e_a = \mu+e_b,$
$$\Phi_{b,a}: M^{\lambda}\rightarrow M^{\mu}.$$
\end{lemma}

The following group algebra element will be important to understand for the analysis.
\begin{definition} Let $\textup{Star}_n \in \SS_n$ to be the element 
$$\textup{Star}_n = \sum_{i=1}^n (i,n)$$
\end{definition}

We are now ready to show a key result illustrating the recursive structure of $Q_{n,k}.$
\begin{proposition}
\begin{equation}
    Q_{n+1,k}\circ \Phi_a - \Phi_a\circ Q_{n,k} = \frac{1}{(n+1)^k}\textup{Star}_{n+1}^k \circ \Phi_a
\end{equation}
\end{proposition}

\begin{proof}
Begin by observing that for $w\in W^n,$ we can expand $(Q_{n+1,k}\circ \Phi_a)(w)$ as
\begin{align*}
    &\,\sum_{1\le i_1,\cdots,i_k\le n+1} \frac{1}{(n+1)^k}(n+1;i_1,\cdots,i_k)(w\cdot a) + \sum_{1\le j\le n}\sum_{1\le i_1,\cdots,i_k\le j} \frac{1}{j^k}(j;i_1,\cdots,i_k)(w\cdot a)\\
    =& \sum_{1\le i_1,\cdots,i_k\le n+1} \frac{1}{(n+1)^k}(n+1;i_1,\cdots,i_k)(w\cdot a) + \sum_{1\le j\le n}\sum_{1\le i_1,\cdots,i_k\le j} \frac{1}{j^k}\Phi_a((j;i_1,\cdots,i_k)(w))\\
    =& \sum_{1\le i_1,\cdots,i_k\le n+1} \frac{1}{(n+1)^k}(n+1;i_1,\cdots,i_k)(w\cdot a) + (\Phi_a\circ Q_{n,k})(w),
\end{align*}
where the primary observation is that we may freely interchange the order of adding a card at the $(n+1)$-th position and permuting the first $n$ cards. Thus,
\begin{equation*}
    (Q_{n+1,k}\circ \Phi_a - \Phi_a\circ Q_{n,k})(w) = \frac{1}{(n+1)^k} \sum_{1\le i_1,\cdots,i_k\le n+1} (n+1;i_1,\cdots,i_k)(w\cdot a), 
\end{equation*}
Since we have the following identity,
\begin{equation*}
   \sum_{1\le i_1,\cdots,i_k\le n+1} (n+1;i_1,\cdots,i_k) = \left(\sum_{1 \le i \le n} (i, n+1)\right)^k = \textup{Star}_{n+1}^k,
\end{equation*}
the desired equation follows.
\end{proof}

\begin{corollary}
\begin{equation}\label{eq:slambda-restriction}
    (Q_{n+1,k}\circ \Phi_a - \Phi_a\circ Q_{n,k})\vert_{S^\lambda} = \frac{1}{(n+1)^k} \textup{Star}_{n+1}^k \circ \Phi_a\vert_{S^\lambda}
\end{equation}
\end{corollary}

The final key component of our proof involves defining the lifting operators, which map eigenvectors of $Q_{n,k}$ to eigenvectors of $Q_{n+1,k}.$ First, we give a useful lemma characterizing the image of the adding operators $\Phi_a.$
\begin{lemma}[Lemma 36 in \cite{raven}]\label{lem:phi-image}
Consider $\lambda\vdash n$ and $\lambda+e_a\vdash n+1$. Then $\Phi_a(S^\lambda)$ is contained in an $\SS_{n+1}-$submodule of $M^{\lambda+e_a}$ isomorphic to $\bigoplus_\mu S^\mu$, where the sum ranges over partitions $\mu$ that can be obtained from $\lambda$ by adding a box in row $i$ for $i\le a.$
\end{lemma}

\begin{definition}
We will define $\pi^\mu:V\rightarrow V$ to be the \textbf{isotypic projection} that projects onto the $S^\mu-$component of $V$. Furthermore, for $\lambda\vdash n, \mu\vdash n+1$, define the operators
$$\kappa_b^{\lambda,\mu} := \pi^\mu \circ \Phi_b: S^\lambda \rightarrow \Phi_b(S^\lambda).$$
As a particular case, define the \textbf{lifting operators}
\begin{equation*}
    \kappa_a^{\lambda,\lambda+e_a}: S^\lambda \rightarrow S^{\lambda+e_a},
\end{equation*}
where the image of $\kappa_a^{\lambda,\lambda+e_a}$ is clear because $\Phi_a(\lambda)$ has a unique $S^{\lambda+e_a}$ component (\Cref{lem:phi-image}).
\end{definition}

We are particularly interested in the lifting operators $\kappa_a^{\lambda,\lambda+e_a},$ as these will allow us to ``lift'' the eigenvectors of $Q_{n,k}$ to $Q_{n+1,k}.$ A key result is that these operators are injective, and thus do not lose any eigenvectors.
\begin{lemma}[Lemma 48 in \cite{raven}]\label{lem:kappa-injective}
Consider $\lambda\vdash n$ where $\lambda+e_a\vdash n+1$. Then the linear operator $\kappa_a^{\lambda,\lambda+e_a}$ is an injective $\SS_n-$module morphism.
\end{lemma}

We follow the work of \cite{raven} to find the eigenvalues of $\text{Star}_{n+1} \circ \kappa_a^{\lambda,\lambda+e_a}$. With this, we can find the eigenvalues of the lifted eigenvectors of $Q_{n+1,k}$.

\begin{lemma}[Lemma 48 of \cite{raven}] For $\lambda\vdash n, a\in[l(\lambda)+1]$, $i \in [a]$, and $\lambda+e_i\vdash n+1,$
$$\textup{Star}_{n+1} \circ \kappa_a^{\lambda,\lambda + e_i} = (2 + \lambda_a - a) \kappa_a^{\lambda,\lambda+e_i} + \sum_{i \le b \le a}\Theta_{b,a} \circ \kappa_b^{\lambda,\lambda+e_i}$$ \end{lemma}

\begin{theorem}\label{thm:recursion}[Theorem 41 of \cite{raven}]
For $\lambda\vdash n, a\in[l(\lambda)+1]$, $i \in [a]$, and $\lambda+e_a\vdash n+1,$
\begin{equation}\label{eq:recursion}Q_{n+1,k}\circ \kappa_a^{\lambda,\lambda+e_i} - \kappa_a^{\lambda,\lambda+e_i}\circ Q_{n,k} =\left(\frac{2 + \lambda_i - i}{n+1}\right)^k\kappa_a^{\lambda,\lambda+e_i},\end{equation}

In particular, if $v\in S^\lambda$ is an eigenvector of $Q_{n,k}$ with eigenvalue $\epsilon$, then $\kappa_a^{\lambda,\lambda+e_a}(v)$ is an eigenvector of $Q_{n+1,k}$ with eigenvalue
$$\epsilon + \left(\frac{2 + \lambda_i - i}{n+1}\right)^k.$$
\end{theorem}

\begin{proof}
Let $\mu = \lambda + e_a$. Observe that
$$\pi^{\mu}\circ (Q_{n+1,k}\circ \Phi_a - \Phi_a\circ Q_{n,k})\vert_{S^\lambda} = Q_{n+1,k}\circ \kappa_a^{\lambda,\mu} - \kappa_a^{\lambda,\mu}\circ Q_{n,k},$$
as $\pi^\lambda$ (an $\SS_{n+1}-$module morphism) commutes with $Q_{n+1,k}.$ For similar reasons, applying $\pi^\mu$ to the other side of (\ref{eq:slambda-restriction}) gives
$$\pi^\mu \circ \textup{Star}_{n+1} \circ \Phi_a\vert_{S^\lambda} = \textup{Star}_{n+1} \circ \kappa_a^{\lambda,\mu}.$$ 
The previous lemma implies the theorem for $i = a$. Apply the switching operator $\Theta_{i,a}$ to the following equation,
$$\text{Star}_{n+1} \circ \kappa_i^{\lambda,\mu} = (2+\lambda_i -i) \kappa_i^{\lambda,\mu}.$$
Since $\Theta_{i,a}$ is a module morphism, the left hand side becomes
\begin{equation*} \begin{split} \text{Star}_{n+1} \circ \pi^\mu \circ \Theta_{i,a} \circ \Phi_i &= \text{Star}_{n+1} \circ \pi^\mu \circ (\Phi_i \circ \Theta_{i,a} - \Phi_a) \\
&= \text{Star}_{n+1} \circ \kappa_i^{\lambda,\mu} \circ \Theta_{i,a} - \text{Star}_{n+1} \circ \kappa_a^{\lambda,\mu}.
\end{split}\end{equation*}
Writing the right hand side in a similar manner,
\begin{equation*} \begin{split} (2+\lambda_i -i) \pi^\mu \circ \Theta_{i,a} \circ \Phi_i &= (2 + \lambda_i - i)\pi^\mu \circ \Phi_i \circ \Theta_{i,a} - (2 + \lambda-i - i) \pi^\mu \circ \Phi_a\\
& = (2 + \lambda_i - i)\kappa_i^{\lambda,\mu} \circ \Theta_{i,a} - (2 - \lambda_i - i) \kappa_a^{\lambda,\mu}.\end{split}\end{equation*}
Equation \ref{eq:recursion} follows from comparing the two sides. \par
By applying what we have shown inductively,
\begin{equation*}\begin{split}Q_{n+1,k}\circ \kappa_a^{\lambda,\mu} - \kappa_a^{\lambda,\mu}\circ Q_{n,k} &= \textup{Star}^k_{n+1} \circ \kappa_a^{\lambda,\mu} \\
&= (2 + \lambda_i - i)\textup{Star}^{k-1}_{n+1} \circ \kappa_a^{\lambda,\mu}\\
&= (2+ \lambda_i - i)^k \kappa_a^{\lambda,\mu}.\end{split}\end{equation*}
\end{proof}

\textsc{Proof of \Cref{thm:eig-syt}.} We now prove \Cref{thm:eig-syt}. We do this by explicitly finding the eigenvalues of $Q_{n,k},$ indexed by the standard Young tableaux of size $n$.

We now show how the eigenvalues of $Q_{n+1,k}$ are obtained from those of $Q_{n,k}$. For $\mu\vdash n+1,$ the branching rules of $S_n$ tell us that
$$\Res_{\SS_n}^{\SS_{n+1}}(S^\mu) = \bigoplus_{\substack{\lambda\vdash n\\\lambda\subset \mu}} S^{\lambda}.$$ 
Now for any $\lambda\vdash n$ such that $\lambda\subset \mu,$ there is some $a$ for which $\lambda + e_a = \mu.$ From \Cref{lem:kappa-injective}, $\kappa_a^{\lambda,\lambda+e_a}$ ``lifts'' a basis of eigenvectors of $Q_{n,k}$ to a basis of eigenvectors of $Q_{n+1,k}$ in $\Res_{\SS_n}^{\SS_{n+1}}(S^\mu)$. As $\Res_{\SS_n}^{\SS_{n+1}}(S^\mu)$ is equal to $S^{\mu}$ as a vector space, we can find a basis of $S^{\mu}$ by lifting eigenvectors of $Q_{n,k}$ for all $\lambda\vdash n$ such that $\lambda\subset \mu.$

This shows how to recursively construct the eigenvalues of $Q_{n,k}.$ First observe that $Q_{1,k}$ has the single eigenvalue $1$ corresponding to $S^{(1)},$ which is of dimension $1$. Then for $\lambda\vdash n,$ each eigenvalue of $Q_{n,k}$ in $S^\lambda$ corresponds to a sequence of partitions $\emptyset = \lambda^{(0)}, \lambda^{(1)}, \cdots, \lambda^{(n)} = \lambda,$ where $\lambda^{(i+1)}$ is obtained from $\lambda^{(i)}$ by adding one box to row $a^{(i)}.$ As we are working with standard tableau, the entry in the box $\lambda^{i+1} \setminus \lambda^i$ must be $n+1$, with coordinate $(i,\lambda_i+1)$. The resulting eigenvalue, by repeated application of \Cref{thm:recursion}, is equal to
\begin{equation}\label{eq:eig-Q}\sum_{(i,j) \in T} \left(\frac{1 + j - i}{T(i,j)}\right)^k.\end{equation}
Moreover, any such sequence of partitions corresponds uniquely to a standard Young tableau $T$ of shape $\lambda$.  Recalling that $M^{(1^n)} \cong \bigoplus_{\lambda\vdash n} d_\lambda S^{\lambda},$ we have that each standard Young tableau $T$ of size $n$ indexes an eigenvalue of $Q_{n,k}$ of multiplicity $d_\lambda$, proving the second assertion of \Cref{thm:eig-syt}. Finally, the first assertion follows from \eqref{eq:eig-Q} since the eigenvalues of $P_{n,k}$ are exactly $\frac{1}{n}$ times the eigenvalues of $Q_{n,k}$.

\bibliographystyle{unsrt}
\bibliography{bib}

\end{document}